\documentclass[reqno, a4paper]{amsart}  

\usepackage{graphicx}
\usepackage[utf8]{inputenc}
\usepackage{epsfig} 
\usepackage{amsmath}
\usepackage{amssymb}
\usepackage{amsthm} 
\usepackage{url}
\usepackage{longtable}
\usepackage{subfig}
\usepackage{algorithmic}
\usepackage{nicefrac}
\usepackage[colorlinks=true,
         linkcolor=black,
         citecolor=black,
         filecolor=black,
         urlcolor=black,
         bookmarks=true,
         bookmarksopen=true,
         bookmarksopenlevel=3,
         plainpages=false,
         pdfpagelabels=true,
         pdftitle={Relaxation Methods for Mixed-Integer Optimal Control of PDEs}
]{hyperref}

\usepackage{graphicx}

\usepackage[utf8]{inputenc}
\usepackage{epsfig} 
\usepackage{amsmath}
\usepackage{amssymb}
\usepackage{url}
\usepackage{longtable}
\usepackage{subfig}
\usepackage{algorithmic}
\usepackage{nicefrac}
\usepackage{mathrsfs}
\usepackage{eufrak}
\usepackage{empheq}
\usepackage{float}

\newfloat{floatingalg}{t}{lop}

\newcommand{\uM}{\bar{M}}

\newcommand{\Uad}{U_{\text{ad}}}

\newcommand{\Vad}{V_{\text{ad}}}

\newcommand{\Lcal}{\mathcal{L}}
\newcommand{\Gcal}{\mathcal{G}}
\newcommand{\Scal}{\mathcal{S}}

\renewcommand{\:}{\mathcal{\colon}}
\newcommand{\NN}{\mathbb{N}}

\newcommand{\RR}{\mathbb{R}}

\DeclareMathOperator{\esssup}{ess\,sup}

\DeclareMathOperator{\cco}{\overline{co}}

\DeclareMathOperator{\cvto}{\curvearrowright}

\newcommand{\rel}{\text{rel}}

\newcommand{\pw}{\text{pw}}
\newcommand{\Id}{\mathrm{Id}}
\newcommand{\PC}{\text{PC}}

\newcommand{\MIP}{MIOCP}

\theoremstyle{plain}
\newtheorem{theorem}{Theorem}
\newtheorem{algorithm}[theorem]{Algorithm}
\newtheorem{proposition}[theorem]{Proposition}
\newtheorem{corollary}[theorem]{Corollary}
\newtheorem{lemma}[theorem]{Lemma}
\theoremstyle{remark}
\newtheorem{definition}[theorem]{Definition}
\newtheorem{remark}[theorem]{Remark}
\newtheorem{example}[theorem]{Example}

\makeatletter
\newif\ifGPblacktext
\GPblacktexttrue

\newif\ifGPcolor
\GPcolortrue

\newif\ifGPblacktext
\GPblacktexttrue
    
\let\gplgaddtomacro\g@addto@macro
\makeatother      

\newcommand{\gplbacktext}{}
\newcommand{\gplfronttext}{}

\begin{document}

\title[Relaxation Methods for Mixed-Integer Optimal Control of PDEs]{Relaxation Methods for Mixed-Integer Optimal Control of Partial Differential Equations}

\author[FALK M. HANTE AND SEBASTIAN SAGER]{Falk M. Hante and Sebastian Sager}

\thanks{F.M. Hante is with Department of Mathematics, University of Erlangen-Nuremberg, Cauerstra\ss e 11, 91058 Erlangen, Germany. Phone: +49 (0)9131 85-67128, \url{hante@math.fau.de}.\\
\indent S. Sager is with Institute of Mathematical Optimization, Otto-von-Guericke University, Universit\"atsplatz 2, 39106 Magdeburg, Germany. Phone: +49 (0)391 6718745, \url{sager@ovgu.de}}

\date{November 22, 2012}

\begin{abstract} 
We consider integer-restricted optimal control of systems governed by abstract 
semilinear evolution equations. This includes the problem of
optimal control design for certain distributed parameter systems endowed with 
multiple actuators, where the task is to minimize costs associated with 
the dynamics of the system by choosing, for each instant in time, one of the 
actuators together with ordinary controls. We consider relaxation
techniques that are already used successfully for mixed-integer 
optimal control of ordinary differential equations. 
Our analysis yields sufficient conditions such that the optimal value and the optimal state 
of the relaxed problem can be approximated with arbitrary precision by a control satisfying the
integer restrictions. The results are obtained by semigroup
theory methods. The approach is constructive and gives rise to a 
numerical method. We supplement the analysis with numerical experiments. 
\end{abstract}

\maketitle

\section{Introduction and Problem Formulation}
The factoring of decision processes interacting with continuous evolution 
plays an important role in model-based optimization for many applications.
For example, when optimally controlling chemical processes there are often both continuous decisions
such as inlet and outlet flows, as well as discrete decisions, such as the operation of on--off
valves and pumps that may redirect flows within the reactor, \cite{Sager2005,Kawajiri2006b}.
Such mixed-integer optimal control problems are therefore 
studied in different communities with different approaches.  Most of these approaches address 
problems that are governed by systems of ordinary differential equations in Euclidean spaces, 
see \cite{Sager2009} for a survey on this topic. 

Total discretization of the underlying system obviously leads to typically large mixed-integer 
nonlinear programs. Hence, relaxation techniques have become an integral part of efficient 
mixed-integer optimal control algorithms, either in the context of branch-and-bound type methods or, more directly, by means of nonlinear 
optimal control methods combined with suitable rounding strategies. An important result  
is that the solution of the relaxed problem can be approximated with arbitrary precision by a solution 
fulfilling the integer requirements \cite{SagerBockDiehl2011}.

In this paper, we extend such relaxation techniques to problems that are governed by certain systems of 
partial differential equations. Motivating applications are for example to switch between reductive 
and oxidative conditions in order to maximize the performance in a monolithic catalyst
\cite{vanSintAnnalandKuipersSwaaij2001}, port switching in chromatographic separation processes \cite{Engell2005,Kawajiri2006b}, or to optimize switching control within 
photochemical reactions \cite{SakataJacksonMaoMarriott2008}. Our problem setting also includes the 
switching control design in the sense that systems are equipped with multiple actuators and the 
optimizer has to choose one of these together with ordinary controls for each
instant in time.

Concerning systems involving partial differential equations, such switching control design has already been 
studied using several techniques: In \cite[Chapter~8]{LiYong1995} optimal switching controls are constructed 
for systems governed by abstract semilinear evolution equations by combining ideas from dynamic programming and 
approximations of the value function using viscosity solutions of the Hamilton-Jacobi-Bellman equations. 
Switching boundary control for linear transport equations using switching time sensitivities has been studied 
in \cite{HanteLeugering2009}. Exemplary for the heat equation and based on variational methods, the controllability 
in case of switching among several actuators has been considered in \cite{Zuazua2011} and null-controllability 
for the one-dimensional wave equation with switching boundary control has been considered in \cite{Gugat2008}. 
Based on linear quadratic regulator optimal control techniques and enumeration of the integer values for a fixed 
time discretization, optimal switching control of abstract linear systems has been considered in \cite{IftimeDemetriou2009}. 

Our approach is complementary to the above, as we break the computationally very expensive 
combinatorial complexity of the problem by relaxation. This comes at the downside of providing only a suboptimal 
solution and possibly at the price of fast switching but, as we will see, with arbitrary small integer-optimality gap,  
depending on discretization, and extensions to limit the number of switching.

We will be concerned with the following problem of mixed-integer optimal control: Minimize a cost functional
\begin{equation}\label{eq:cost}
 J = \phi(z(t_f)) + \int_0^{t_f} \psi(z(t),u(t))\,dt
\end{equation}
over trajectories $z \in X_{[0,t_f]} \subset \{z\: [0,t_f] \to X\}$ and control functions 
$u \in U_{[0,t_f]} \subset \{u \: [0,t_f] \to U\}$ and $v \in V_{[0,t_f]} \subset \{v \: [0,t_f] \to V\}$
subject to the constraints that $z$ is a mild solution of the operator
differential equation
\begin{equation}\label{eq:sysuc}
\left\{\begin{aligned}
 &\dot{z}(t) = Az(t) + f(t,z(t),u(t),v(t)), \quad t \in (0,t_f]\\
 &z(0)=z_0 \in X
\end{aligned}\right.
\end{equation}
and that the control functions satisfy
\begin{equation}\label{eq:constraintuc}
u(t) \in \Uad \subset U,~v(t) \in \Vad \subset V,\quad t\in[0,t_f]
\end{equation}
where $X$, $U$ and $V$ are Banach spaces, $X_{[0,t_f]}$, $U_{[0,t_f]}$ and $V_{[0,t_f]}$ are normed linear spaces, 
$A\: D(A) \to X$ is the infinitesimal generator of a strongly continuous semigroup $\{T(t)\}_{t \geq 0}$ 
on $X$, $t_f \geq 0$ is a fixed real number, $f\: [0,t_f]  \times X \times U \times V
\to X$, $\phi\: X \to \RR$ and $\psi\: X \times U \to \RR$ are given functions,
$\Uad$ is some subset of $U$ and $\Vad$ is a finite subset of $V$. This setting is for the most part
classical, except for the assumptions on $\Vad$.

We will refer to the above infinite-dimensional dynamic optimization
problem as mixed-integer optimal control problem, short (\MIP),
and to the control function $[u,v]$ as a mixed-integer control. 
This accounts for the fact that we do not impose restrictions on the set 
$\Uad \subset U$ while we can always identify the finite set $\Vad \subset V$ 
of the feasible control values for $v$ with a finite number of integers
\begin{equation}
\Vad = \{v^1,\ldots,v^N\} \simeq \{1,\ldots,N\}.
\end{equation}
Moreover, the operator differential equation \eqref{eq:sysuc} is an abstract
representation of certain initial-boundary value problems governed by linear
and semilinear partial differential equations, see, e.\,g., \cite{Pazy1983}.

The existence of an optimal solution of the problem (\MIP) depends, inter alia,
on the spaces $X_{[0,t_f]}$, $U_{[0,t_f]}$ and $V_{[0,t_f]}$ where we seek 
$z\: [0,t_f] \to X$, $u\: [0,t_f] \to U$ and $v\: [0,t_f] \to V$, respectively. Common choices 
are, for $U_ {[0,t_f]}$, the spaces of square integrable ($L^2$), piecewise $k$-times differentiable 
($C^k_\pw$) or (piecewise) $k$-times weakly differentiable ($H^k_\pw$) functions $u \: [0,t_f] \to U$
and, for $V_{[0,t_f]}$, the spaces of essentially bounded ($L^\infty$) or piecewise constant 
($\PC$) functions $v\: [0,t_f] \to V$. We defer these considerations by assuming later that 
there exists an optimal solution of a related (to a certain extent convexified and relaxed) 
optimal control problem and present sufficient conditions guaranteeing that the solution of 
the relaxed problem can be approximated with arbitrary precision by a solution satisfying 
the integer restrictions. We do only assume that $X_{[0,t_f]} \subset C([0,t_f];X)$, 
$U_{[0,t_f]}\subset L^1(0,t_f;U)$ and $V_{[0,t_f]} \subset L^1(0,t_f;V)$ to ensure that certain 
quantities in problem (\MIP) are well-defined.

This relaxation method becomes most easily evident from writing the problem (\MIP)
using a differential inclusion, that is, minimize \eqref{eq:cost} subject to the 
constraints that $z$ is a solution of 
\begin{equation}\label{eq:di}       
  \left\{\begin{aligned}
       &\dot{z}(t) \in Az(t) + \{f(t,z(t),u(t),v^i) : v^i \in \Vad\},
       \quad t \in (0,t_f]\\ 
       &z(0)=z_0.
   \end{aligned}\right. 
\end{equation}       
and $u$ satisfies
\begin{equation*}
       u(t) \in \Uad,~\quad~t\in[0,t_f].
\end{equation*}
It is well known that, under certain technical assumptions, the solution set of
\eqref{eq:di} is dense in the solution set of the convexified differential 
inclusion
\begin{equation}\label{eq:diconvex}       
  \left\{\begin{aligned}
       &\dot{z}(t) \in Az(t) + \cco\{f(t,z(t),u(t),v^i) : v^i \in \Vad\},
       \quad t \in (0,t_f]\\ 
       &z(0)=z_0,
   \end{aligned}\right. 
\end{equation}
where $\cco$ denotes the closure of the convex hull. This is proved in 
\cite{Frankowska1990} for the case when $X$ is a separable Banach space 
and in \cite{deBlasiPianigiani1999} for non-separable Banach spaces. While
these results rely on powerful selection theorems, our main contribution
is a constructive proof based on discretization, giving rise to a numerical 
method at the prize of additional regularity assumptions.

We will see that the advantage of such a relaxation method is that the
convexified problem, using a particular representation of \eqref{eq:diconvex}, 
falls into the class of optimal control problems with partial differential
equations without integer-restrictions. The already known theory, in particular 
concerning existence, uniqueness and regularity of optimal solutions as well as
numerical considerations such as sensitivities, error analysis for finite element
approximations, etc., can thus be carried over to the mixed-integer problem
under consideration here. We also discuss the possibility to include state-constraints
for example enforcing time-periodicity constraints
\begin{equation}
 z(t_f)=z(0),
\end{equation}
as occurring in chromatographic separation processes \cite{Engell2005,Kawajiri2006b}.
The disadvantage of this approach is that we target at a solution that is only suboptimal 
(though with arbitrary precision) and that switching costs, a standard regularization of 
mixed-integer problems to prevent chattering solutions, or additional combinatorial constraints 
can lead to larger optimality/feasibility gaps. Nevertheless, we will show how a-priori bounds for such a gap
can be obtained when constraints on the number of switches are incorporated.

The framework we use for the analysis here will be semigroup theory. Recall
that for given $z_0 \in X$ and given control functions $u,v$, the mild solution
of the state equation \eqref{eq:sysuc} is given by a function $z \in C(0,t_f;X)$ 
satisfying the variation of constants formula
\begin{equation}\label{eq:vc}
z(t) = T(t)z_0 + \int_0^t T(t-s)f(s,z(s),u(s),v(s))\,ds,\quad 0 \leq t \leq t_f
\end{equation} 
in the Lebesgue-Bochner sense. This abstract setting covers in particular the
usual setup for weak solutions of linear parabolic partial differential equations 
with distributed control on reflexive Banach spaces where $A$ arises from a
time-invariant variational problem, see \cite[Section~1.3]{BensoussanDaPratoDelfourMitter1992}.

We include in our analysis explicitly the possibility to approximate the state equation \eqref{eq:sysuc} 
and say that $\tilde{z} \in X_{[0,t_f]}$ is an $\varepsilon$-accurate solution
of \eqref{eq:sysuc} if $z$ is the mild solution and $\|\tilde{z}(t)-z(t)\|_X \leq \varepsilon$ for
all $t \in [0,t_f]$. 
Accordingly, we define for a given $\varepsilon \geq 0$ the set of $\varepsilon$-admissible solutions for (\MIP) as
\begin{equation}\label{eq:Xidef}
 \begin{aligned}
 &\Xi_\varepsilon := \bigr\{(u,v,z) \in U_{[0,t_f]} \times V_{[0,t_f]} \times X_{[0,t_f]} : u,v~\text{satisfy}~\eqref{eq:constraintuc}\\
 &\qquad\qquad~\text{and}~z=z(u,v)~\text{is an}~\varepsilon\text{-accurate solution of}~\eqref{eq:sysuc}\bigr\}.
 \end{aligned}
\end{equation}
Further, we denote throughout the paper by $H^{1}_\pw(0,t_f;X)$ the space 
of $X$-valued functions defined on the interval $[0,t_f]$ and being piecewise once-weakly 
differentiable with a piecewise defined weak derivative that is square-integrable in the 
Lebesgue-Bochner sense. Consistently, we denote by $C^{0,\vartheta}_{\pw}(0,t_f;X)$ the space 
of $X$-valued functions defined on the interval $[0,t_f]$ being piecewise Hölder-continuous with 
a Hölder-constant $\vartheta$. In both constructions, piecewise means that there exists a finite partition of the 
interval $[0,t_f]$
\begin{equation}\label{eq:piecewisedef}
  0=\tau_0<\tau_1<\tau_2<\ldots<\tau_{K_{t_f}}<\tau_{K_{t_f}+1}=t_f
\end{equation}
so that the function has the respective regularity on all intervals 
$[\tau_k,\tau_{k+1})$, $k=0,\ldots,K_{t_f}$. We denote by $\|\cdot\|_X$ the norm on $X$ and
by $\|\cdot\|_{\Lcal(X)}$ the operator norm induced by $\|\cdot\|_X$. Further, we denote by
$\|\cdot\|_{U_{[0,t_f]}}$ and $\|\cdot\|_{V_{[0,t_f]}}$ the norm of $U_{[0,t_f]}$ and $V_{[0,t_f]}$,
respectively. For simplicity of notation, we also define $T(-t)=\Id$ for all $t>0$, $\Id$ denoting the identity on $X$.

The paper is organized as follows. In Section~\ref{sec:relax}, we present
details of the relaxation method and the main results concerning estimates of
the approximation error. 
In Section~\ref{sec:cconstraints}, we discuss extensions of the method to 
incorporate certain combinatorial constraints.
In Section~\ref{sec:examples}, we discuss applications for linear
and semilinear equations and present numerical results for the heat equation
with spatial scheduling of different actuators and a semilinear
reaction-diffusion system with an on-off type control. In Section~\ref{sec:final}, we
conclude with some additional remarks and point out open problems.

\section{Relaxation Method}\label{sec:relax}
Consider the following problem involving a particular representation of
the convexified differential inclusion \eqref{eq:diconvex} 
\begin{subequations} \label{eq:relaxedconvexifiedproblem}
\begin{empheq}[left=\empheqlbrace]{align}
       &\underset{{[\omega,\alpha,y] \in U_{[0,t_f]} \times \tilde{V}_{[0,t_f]} \times X_{[0,t_f]}}}{\text{minimize}}~J = \phi(y(t_f)) + \int_0^{t_f} \psi(y(t),\omega(t))\,dt\quad
       \text{s.\,t.} \label{eq:relaxedconvexifiedproblemA}\\ 
       &\quad \dot{y}(t) = A y(t) + \sum_{i=1}^N \alpha_i(t)
       f(t,y(t),\omega(t),v^i),\quad t \in  (0,t_f]\label{eq:relaxedconvexifiedproblemB} \\ 
       &\quad y(0)=y_0:=z_0\label{eq:relaxedconvexifiedproblemC}\\
       &\quad \vphantom{\sum_{i=1}^N} \omega(t) \in \Uad,\quad~t\in[0,t_f]\label{eq:relaxedconvexifiedproblemD}\\
       &\quad \alpha(t)=(\alpha_1(t),\ldots,\alpha_N(t)) \in [0,1]^N,\quad~t\in[0,t_f]\label{eq:relaxedconvexifiedproblemE}\\
       &\quad \sum_{i=1}^N \alpha_i(t) = 1,\quad~t\in[0,t_f]\label{eq:relaxedconvexifiedproblemF},
\end{empheq}
\end{subequations}
where we write $\omega,\alpha,y$ for $u,v,z$, respectively, to emphasize the relaxation of the original problem (\MIP).
Here, $\tilde{V}_{[0,t_f]}=\{\alpha\: [0,t_f] \to \RR^N : \|\alpha\|_{V_{[0,t_f]}} < \infty\}$.

\begin{floatingalg}
\rule{\textwidth}{0.4pt}
\begin{algorithm}\label{algo:iter}\mbox{} \normalfont
\begin{algorithmic}[1] 
  \STATE Choose a time discretization grid 
  $\Gcal^0 = \{0 = t^0_0 < t^0_1 < \cdots < t^0_{n^0} = t_f\}$, a sequence of non-negative 
  accuracies $\{\varepsilon^k\}_{k \in \NN}$ and some termination tolerance $\varepsilon>0$. 
  Set $k=0$.
  \LOOP
	    \STATE Find an $\varepsilon^k$-accurate optimal solution $[\omega^k,\alpha^k,y^k]$ of 
	    the relaxed problem \eqref{eq:relaxedconvexifiedproblem} and set $J^k_\rel:=J(\omega^k,\alpha^k,y^k)$.
	    \STATE If $\alpha^k \in \PC(0,t_f;\{0,1\}^N)$ and $\varepsilon^k \leq \varepsilon$, 
	    then set $v^k(t)=\sum_{i=1}^N \alpha_i^k(t) v^i$ and $z^k(t)=y^k(t)$ for $t \in [0,t_f]$ and STOP.
            \STATE Using $\Gcal^k$ and $\alpha^k$, define a piecewise constant function
            $\beta^k=(\beta^k_1,\ldots,\beta^k_N): [0,t_f] \to \{0,1\}^N$ by
            \begin{equation}\label{eq:omegadef}\tag{A$_1$}
            \beta^k_i(t) = p^k_{i,j},\quad t\in [t^k_j,t^k_{j+1}),~i=1,\ldots,N,~j=0,\ldots,n^k-1
            \end{equation}
            where for all $i=1,\ldots,N$, $j=0,\ldots,n^k-1$
            \begin{equation}\label{eq:pdef}\tag{A$_2$}
            \begin{aligned}
            &p^k_{i,j} = \begin{cases}1~&\text{if}~\left(\hat{p}^k_{i,j} \geq \hat{p}^k_{l,j}~\forall~l \in \{1,\ldots,N\}\setminus\{i\}\right) \text{and}\\
            &~\left(i<l~\forall~l \in \{1,\ldots,N\}\setminus\{i\}~\text{:}~\hat{p}^k_{i,j}=\hat{p}^k_{l,j}\right)\\
            0~&\text{else}\end{cases}\\
            &\hat{p}^k_{i,j}= \int_0^{t^k_{j+1}} \alpha^k_i(\tau)\,d\tau - \sum_{l=0}^{j-1} p^k_{i,l} (t^k_{l+1}-t^k_{l}).
            \end{aligned}
            \end{equation}
            \STATE Set $u^k(t)=\omega^k(t)$, $v^k(t) = \sum_{i=1}^N \beta_i^k(t) v^i$ for $t \in [0,t_f]$ and $J^k=\phi(z^k(t_f)) + \int_0^{t_f}
            \psi(z^k(t),\omega^k(t))\,dt$ where $z^k$ is an $\varepsilon^k$-accurate solution of
            \begin{equation}\label{eq:sysforward}\tag{A$_3$}
            \dot{z}(t) = Az(t) + \sum_{i=1}^N \beta^k_i(t)
            f(t,z(t),\omega^k(t),v^i),\quad t \in  (0,t_f],\quad z(0)=z_0.
            \end{equation} 
            \STATE If $|J^k_\rel-J^k| \leq \frac{\varepsilon}{2}$ and $\varepsilon^k \leq \frac{\varepsilon}{2}$ then STOP.  
            \STATE Choose $\Gcal^{k+1}=\{0 = t^{k+1}_0 < t^{k+1}_1 < \cdots
            < t^{k+1}_{n^{k+1}} = t_f\}$ such that $\Gcal^{k}\subset
            \Gcal^{k+1}$ and set $k=k+1$.
  \ENDLOOP
  \STATE Set $u^*(t)=\omega^k(t)$, $v^*(t) = v^k(t)$ and $z^*(t)=z^k(t)$ for $t \in [0,t_f]$.
\end{algorithmic} 
\end{algorithm}
\rule{\textwidth}{0.4pt}
\end{floatingalg}

Observe that the control functions $\alpha_i$ take values on the full interval $[0,1]$,
but that any optimal controls $[\omega^*,\alpha^*]$ of
\eqref{eq:relaxedconvexifiedproblem} yields an optimal mixed-integer control 
\begin{equation}\label{eq:controlbij}
[u^*,v^*]:=[\omega^*,\sum_{i=1}^N \alpha^*_i v^i] 
\end{equation}
of problem (\MIP) if $\alpha^*(t) \in \{0,1\}^N$ for almost every $t \in (0,t_f)$. However, it is not 
very difficult to construct examples where $\alpha^*(t) \in (0,1)^N$ for $t$ on some interval
of positive measure. It is only in some special cases where optimality of
$\alpha^*(t) \in [0,1]^N$ implies that $\alpha^*$ takes only values on the
boundary of its feasible set. For examples where this property, known as the bang-bang principle,
can be verified in the context of partial differential equations, 
see \cite[Section~3.2.4]{Troeltzsch2010} and the references therein. Moreover, the relaxed problem
\eqref{eq:relaxedconvexifiedproblem} can in most applications only be solved approximately.

Therefore, consider the following hypothesis.
\begin{list}{\labelitemi}{\itemsep=1em\leftmargin=1.5em}
  \item[{\normalfont (H$_0$)}] Problem \eqref{eq:relaxedconvexifiedproblem} has an optimal
  solution in $U_{[0,t_f]} \times \tilde{V}_{[0,t_f]} \times X_{[0,t_f]}$.
\end{list}
Under this assumption, we will accept $\varepsilon$-accurate optimal solutions of the relaxed problem \eqref{eq:relaxedconvexifiedproblem} and
propose in Algorithm~\ref{algo:iter} an iterative procedure to obtain from these solutions
a mixed-integer control taking values in $\Uad \times \{0,1\}^N$. We then show in Theorem~\ref{thm:convergence} 
under certain technical assumptions that the optimal value and the optimal state of problem \eqref{eq:relaxedconvexifiedproblem} can 
be approximated by the proposed procedure with arbitrary precision. To make our terminology precise, we 
include the following definition.

\begin{definition} Given some $\varepsilon \geq 0$, we say that $[\omega^\varepsilon,\alpha^\varepsilon,y^\varepsilon] \in U_{[0,t_f]} \times \tilde{V}_{[0,t_f]} \times X_{[0,t_f]}$ 
is an \emph{$\varepsilon$-accurate optimal solution of \eqref{eq:relaxedconvexifiedproblem}} if
$y^\varepsilon$ is an $\varepsilon$-accurate solution of \eqref{eq:relaxedconvexifiedproblemB}--\eqref{eq:relaxedconvexifiedproblemC}
with $\omega=\omega^\varepsilon$ and $\alpha=\alpha^\varepsilon$, the constraints \eqref{eq:relaxedconvexifiedproblemD}--\eqref{eq:relaxedconvexifiedproblemF}
hold with $\omega=\omega^\varepsilon$ and $\alpha=\alpha^\varepsilon$ and
\begin{equation}
J(\omega^\varepsilon,\alpha^\varepsilon,y^\varepsilon) \leq \inf_{\omega,\alpha,y~\text{s.t. \eqref{eq:relaxedconvexifiedproblemB}--\eqref{eq:relaxedconvexifiedproblemF}}} J(\omega,\alpha,y) + \varepsilon.
\end{equation}
\end{definition}

So any $\varepsilon$-accurate optimal solution of the relaxed problem \eqref{eq:relaxedconvexifiedproblem} is admissible with respect to
the control constraints on $\alpha$ and $\omega$, $\varepsilon$-close admissible with respect to the state variable $y$ and the corresponding
value of the cost function is $\varepsilon$-close to the optimal value of problem \eqref{eq:relaxedconvexifiedproblem}. 

Now, consider Algorithm~\ref{algo:iter} on page~\pageref{algo:iter} to obtain a
mixed-integer control for problem (\MIP). The main result is the following.

\begin{theorem}\label{thm:convergence}
Assuming {\normalfont (H$_0$)}, let $[\omega^*,\alpha^*,y^*]$ denote an optimal solution of the 
relaxed problem \eqref{eq:relaxedconvexifiedproblem} and assume that the following
assumptions hold true.
\begin{list}{\labelitemi}{\itemsep=1em\leftmargin=1em}
  \item[{\normalfont(H$_1$)}] The functions $\phi$, $\psi$, and $f$ satisfy the Lipschitz-estimates
\begin{equation*}
\begin{aligned}
|\phi(y_1)-\phi(y_2)| &\leq \eta \|y_1 -y_2\|_X,\\
|\psi(y_1,\omega_1)-\psi(y_2,\omega_2)| &\leq \xi (\|y_1 - y_2\|_X + \|\omega_1-\omega_2\|_U),\\
\left\|f(t,y_1,\omega_1,v^i)-f(t,y_2,\omega_2,v^i)\right\|_X &\leq L (\|y_1 -  y_2\|_X + \|\omega_1 - \omega_2\|_U),
\end{aligned}
\end{equation*} 
for all $y_1,y_2 \in X,~\omega_1,\omega_2 \in \Uad$, $t \in [0,t_f]$ and $i=1,\ldots,N$ 
with positive constants $\eta$, $\xi$ and $L$.
\item[{\normalfont (H$_2$)}] For all $i=1,\ldots,N$ and $t \in [0,t_f]$ the function 
  \begin{equation*}
  s \mapsto T(t-s)f(s,y^*(s),\omega^*(s),v^i) 
  \end{equation*}
  is in $H^1_{\pw}(0,t_f;X)$
  and there exists a positive constant $C_i$ such that
  \begin{equation*}
  \left\|\frac{d}{ds}(T(t-s)f(s,y^*(s),\omega^*(s),v^i))\right\|_X \leq
  C_i\quad\text{a.\,e. in}~0<s<t<t_f.
  \end{equation*}
  Let $C=\sum_{i=1}^N C_i$.
  \item[{\normalfont (H$_3$)}] For all $i=1,\ldots,N$, there exists a positive 
  constant $M_i$ such that
  \begin{equation*}
    \sup_{t\in[0,t_f]} \|f(t,y^*(t),\omega^*(t),v^i)\|_X \leq M_i.
  \end{equation*}
  Let $M=\sum_{i=1}^{N} M_i$.
  \item[{\normalfont(H$_4$)}] The solution $[\omega^*,\alpha^*,y^*]$ is stable in the sense that there exists a 
  positive constant $C_J$ such that 
  \begin{equation}
  \begin{aligned}
   &\|\omega^*(t)-\omega(t)\|_U + \|\alpha^*(t)-\alpha(t)\|_{\RR^N} \leq \\
   &\qquad C_J |J(\omega^*,\alpha^*,y(\omega^*,\alpha^*))-J(\omega,\alpha,y(\omega,\alpha))|~\text{a.\,e. in}~(0,t_f)
   \end{aligned}
  \end{equation}
  for $[\omega,\alpha]$ in some neighborhood of $[u^*,\alpha^*]$ in $U_{[0,t_f]} \times \tilde{V}_{[0,t_f]}$.
\end{list}
Moreover, assume that $\varepsilon^k \to 0$ and that the sequence $\{\Gcal^{k}\}_k$ in Algorithm~\ref{algo:iter} is such that
$\Delta t^k \to 0$ with
\begin{equation}\label{eq:deltatdef}
\Delta t^k = \max_{i=1,\ldots,n^k} \{t^k_{i}-t^k_{i-1}\}.
\end{equation} 
Define the constants $C_1=(C_J(M+1)(1+\eta+\xi)e^{t_f \bar{M}L}+1)$, $C_2=(M + t_f C)e^{t_f \uM L}$,
$C_3:=(\eta + t_f\xi) C_1 + t_f\xi C_J$ and $C_4:=(\eta+t_f\xi)C_2$, where 
$\uM=\sup_{t \in [0,t_f]} \|T(t)\|_{\Lcal(X)}$ and the constants $\eta$, $\xi$, $M$, $C$, $L$ and $C_J$ are given by hypothesis 
{\normalfont (H$_1$)}--{\normalfont (H$_4$)}. Then, $[u^k,v^k,z^k]$ defined by Algorithm~\ref{algo:iter} is 
in $\Xi_{\varepsilon^k}$ for all $k=0,1,2,\ldots$ and satisfies the estimates
\begin{equation}\label{eq:solest0}
 \|y^*(t)-z^k(t)\|_X \leq C_1\varepsilon^k+C_2(N-1)\Delta t^k,\quad t \in [0,t_f],
\end{equation}
and 
\begin{equation}\label{eq:Jest0}
\begin{aligned}
 &|J(\omega^*,\alpha^*,y^*) - J(u^k,v^k,z^k)| \leq C_3 \varepsilon^k + C_4(N-1) \Delta t^k.
 \end{aligned}
\end{equation}
In particular, Algorithm~\ref{algo:iter} terminates in a finite number of steps with an $\varepsilon$-feasible 
mixed-integer solution $[u^*,v^*,z^*] \in \Xi_{\varepsilon}$ of Problem (\MIP) satisfying the estimate 
\begin{equation}\label{eq:MIOCPoptValueEst}
|J(\omega^*,\alpha^*,y^*) - J(u^*,v^*,z^*)| \leq \varepsilon,
\end{equation}
where $\varepsilon>0$ is chosen arbitrarily in step {\footnotesize\normalfont $1$:}.
\end{theorem}


Before we prove Theorem~\ref{thm:convergence}, we first prove a result saying
that the deviation of two mild solutions in $X_{[0,t_f]}$ equipped with the
uniform norm can be estimated in terms of the absolute value of the integrated difference of two linearly entering 
control functions. This estimate is non-standard and generalizes the result in \cite{SagerBockDiehl2011}
to a Banach space setting, noting that the absolute value of the integrated difference is not a norm and 
in particular not comparable to the $V_{[0,t_f]}$-norm as a natural choice. This estimate, together with an approximation result 
for this integrated difference is the key ingredient in order to prove all a-priori estimates needed for
the proof of Theorem~\ref{thm:convergence}.

\begin{lemma} \label{lem:solutionest} Let $\varepsilon>0$ and $\uM=\sup_{t \in
[0,t_f]} \|T(t)\|_{\Lcal(X)}$. 
Suppose that $[\omega^*,\alpha^*,y^*]$ is a feasible solution of the relaxed problem
\eqref{eq:relaxedconvexifiedproblem} and assume that the hypotheses
{\normalfont (H$_1$)}--{\normalfont (H$_3$)} of Theorem~\ref{thm:convergence} hold true. Let
$\beta=(\beta_1,\ldots,\beta_N) \in L^\infty(0,t_f;[0,1]^N)$ be such 
that
\begin{equation}\label{eq:intdiffest}
 \max_{i=1,\ldots,N}\sup_{t\in[0,t_f]}\left|\int_0^t
 \alpha^*_i(\tau)-\beta_i(\tau)\,d\tau \right|
 \leq \varepsilon
\end{equation}
and let $z$ be the mild solution of \eqref{eq:sysforward} in Algorithm~\ref{algo:iter} with
$\beta^k_i=\beta_i$, $i=1,\ldots,N$, and $\omega^k=\omega^*$. Then
\begin{equation}
 \|y^*(t)-z(t)\|_X \leq \left((M + Ct)e^{\uM L t}\right)\varepsilon,\quad t\in[0,t_f].
\end{equation}
\end{lemma}
\begin{proof}
Fix $t \in [0,t_f]$ and set, for the sake of brevity, 
$\delta(t)=\|y^*(t)-z(t)\|_X$ and $f^i(t,y(t))=f(t,y(t),\omega^*(t),v^i)$. Recalling \eqref{eq:piecewisedef} 
and using hypothesis {\normalfont (H$_2$)}, let $\{\tau_0,\tau_1,\ldots,\tau_{K+1}\}$ be the set of partition points of 
the functions $s \mapsto T(t-s)f^i(s,y^*(s))$ as objects in $H^1_{\pw}(0,t;X)$, $i=1,\ldots,N$, so that 
$\tau_0=0$ and $\tau_{K+1}=t$. From the definition of the mild solutions for
\eqref{eq:relaxedconvexifiedproblem} and \eqref{eq:sysforward}, we have
\begin{equation*}
\delta(t) = \left\| \sum_{i=1}^N \int_0^t
T(t-s)f^i(s,y^*(s))\alpha_i^*(s) -
T(t-s)f^i(s,z(s))\beta_i(s)\,ds \right\|_X.
\end{equation*}
Adding
$0=T(t-s)f^i(s,y^*(s))\beta_i(s)-T(t-s)f^i(s,y^*(s))\beta_i(s)$
under the integral, applying the triangular inequality and rearranging terms 
this yields
\begin{equation*}
\begin{aligned}
\delta(t) &\leq \sum_{i=1}^N \left\| \int_0^t T(t-s)[f^i(s,y^*(s))-
f^i(s,z(s))]\beta_i(s)\,ds\right\|_X \\
& \quad + \sum_{i=1}^N \left\| \sum_{k=0}^K \int_{\tau_k}^{\tau_{k+1}} T(t-s)
f^i(s,y^*(s))[\alpha_i^*(s)-\beta_i(s)] ds\right\|_X.
\end{aligned}
\end{equation*}
Now using integration by parts in the second part, we obtain
\begin{equation*}
\begin{aligned}
\delta(t) &\leq \sum_{i=1}^N \left\| \int_0^t T(t-s)[f^i(s,y^*(s))-
f^i(s,z(s))]\beta_i(s)\,ds\right\|_X \\
& \quad +\sum_{i=1}^N \bigg\| \sum_{k=0}^{K} \bigg(
T(t-\tau_{k+1})f^i(\tau_{k+1},y^*(\tau_{k+1}))\int_{0}^{\tau_{k+1}}
\alpha_i^*(s)-\beta_i(s)\,ds\\ &\quad -
T(t-\tau_k)f^i(\tau_k,y^*(\tau_k))\int_{0}^{\tau_k}\alpha_i^*(s)-\beta_i(s)\,ds\\
& \quad - \int_{\tau_k}^{\tau_{k+1}}
\frac{d}{ds}\left(T(t-s)f^i(s,y^*(s))\right)
\int_0^s\alpha_i^*(\vartheta)-\beta_i(\vartheta)\,d\vartheta\,ds\bigg)
\bigg\|_X.
\end{aligned}
\end{equation*}
Then by rearranging terms, noting that the appearing telescopic sum evaluates as
\begin{equation*}
\begin{aligned}
&\sum_{k=0}^{K} \bigg(T(t-\tau_{k+1})f^i(\tau_{k+1},y^*(\tau_{k+1}))\int_{0}^{\tau_{k+1}}
\alpha_i^*(s)-\beta_i(s)\,ds\\
&\quad -
T(t-\tau_k)f^i(\tau_k,y^*(\tau_k))\int_{0}^{\tau_k}\alpha_i^*(s)-\beta_i(s)\,ds\bigg)\\
& = f^i(t,y^*(t))\int_{0}^{t}\alpha_i^*(s)-\beta_i(s)\,ds
\end{aligned}
\end{equation*} 
because of $\tau_0=0$, $\tau_{K+1}=t$, $T(t-t)=\mathrm{Id}$ and
$\int_{0}^{0}\alpha_i^*(\vartheta)-\beta_i(\vartheta)\,d\vartheta = 0$,
and by applying the triangular inequality this estimate simplifies to
\begin{equation*}
\begin{aligned}
\delta(t) &\leq \sum_{i=1}^N \int_0^t
\left\|T(t-s)\right\|_{\Lcal(X)}\|f^i(s,y^*(s))-
f^i(s,z(s))\|_X|\beta_i(s)|\,ds \\ 
& \quad +\sum_{i=1}^N \left\|f^i(t,y^*(t))\right\|_X
\left|\int_{0}^{t} \alpha_i^*(s)-\beta_i(s)\right|\,ds\\ &\quad +
\sum_{i=1}^N \int_0^t
\left\|\frac{d}{ds}\left(T(t-s)f^i(s,y^*(s))\right)\right\|_X
\left|\int_0^s\alpha_i^*(\vartheta)-\beta_i(\vartheta)\,d\vartheta\right|\,ds.
\end{aligned}
\end{equation*}
Then, by definition of $\delta$, $f^i$ and the constant $\uM$, the definition of the constants $L$,
$C$ and $M$ in hypotheses {\normalfont (H$_1$)}--{\normalfont (H$_3$)}, the assumption \eqref{eq:intdiffest} and the fact 
that $\beta_i(t)\leq 1$, this yields
\begin{equation*}
\begin{aligned}
\delta(t) &\leq \uM L \int_0^t \delta(s)\,ds + M \varepsilon +
Ct\varepsilon.
\end{aligned}
\end{equation*}
Finally, using the Gronwall lemma and rearranging terms, we obtain the desired
estimate
\begin{equation*}
\delta(t) \leq \left((M + Ct)e^{\uM L t}\right)\varepsilon.
\end{equation*}
\qed
\end{proof}

Next, we recall from \cite{SagerBockDiehl2011} the following result on integral approximations. 
\begin{lemma}\label{lem:integeralapprox} Let
$\alpha=(\alpha_1,\ldots,\alpha_N)\: [0,t_f] \to [0,1]^N$ be a measurable
function satisfying $\sum_{i=1}^N \alpha_i(t)=1$ for all $t \in [0,t_f]$.
Define a piecewise constant function $\beta\: [0,t_f] \to \{0,1\}^N$ by
\begin{equation}
            \beta_i(t) = p_{i,j},\quad t\in [t_j,t_{j+1}),~i=1,\ldots,N,~j=0,\ldots,n-1
\end{equation}
where for all $i=1,\ldots,N$, $j=0,\ldots,n-1$, $p_{i,j}$ is defined by \eqref{eq:pdef} in Algorithm~\ref{algo:iter} with $p_{i,j}^k=p_{i,j}$.
Then it holds for $\Delta t = \max_{i=1,\ldots,n} \{t_{i}-t_{i-1}\}$
\begin{enumerate}
  \item $\displaystyle \max_{i=1,\ldots,N} \left| \int_0^t \alpha_i(\tau) -
  \beta_i(\tau)\,d\tau \right| \leq (N-1) \Delta t$ for all $t \in [0,t_f]$,
  \item $\displaystyle \sum_{i=1}^N \beta_i(t) = 1$ for all $t \in [0,t_f]$.
\end{enumerate} 
\end{lemma}
\begin{proof}
 See Theorem~5 of \cite{SagerBockDiehl2011}.
\end{proof}

With the above two results we are now in the position to prove Theorem~\ref{thm:convergence}.

\begin{proof}[Theorem~\ref{thm:convergence}]
Let the assumptions of Theorem~\ref{thm:convergence} hold true. First we show that the sequence
$[u^k,v^k,z^k]$ obtained by Algorithm~\ref{algo:iter} is $\varepsilon^k$-feasible for the problem (\MIP).
We have $u^k=\omega^k \in \Uad$ for all $k$ by construction of $\omega^k$ in step {\footnotesize $3$:}, 
$v^k(t) = \sum_{i=1}^N \beta_i^k(t) v^i$, $t \in [0,t_f]$, by construction in step {\footnotesize $6$:}
so $v^k(t) \in \Vad$, $t \in [0,t_f]$, because $\beta_i^k(t) \in \{0,1\}$ for all $i,k$ and $t \in [0,t_f]$ 
as seen from \eqref{eq:omegadef} and \eqref{eq:pdef}. By construction in step {\footnotesize $6$:},
$z^k$ is an $\varepsilon^k$-accurate solution of \eqref{eq:sysforward}. Thus, recalling \eqref{eq:Xidef}, 
$[u^k,v^k,z^k] \in \Xi_{\varepsilon^k}$ for all $k=0,1,2,\ldots$.

Next, we show \eqref{eq:solest0}. To this end, let $y(\cdot;\omega^k,\beta^k)$, $z(\cdot;\omega^*,\beta^k)=y(\cdot;\omega^*,\beta^k)$
and $y(\cdot;\omega^*,\alpha^k)$ denote the mild solutions of \eqref{eq:relaxedconvexifiedproblemB}, \eqref{eq:relaxedconvexifiedproblemC}
with the respective controls. The stability assumption (H$_4$) and the continuity assumption (H$_1$) implies that
\begin{equation}\label{eq:omegaconv}
\begin{aligned}
&\|\alpha^*(t)-\alpha^k(t)\|_{\RR^N} + \|\omega^*(t) - \omega^k(t)\|_U\\ 
&\qquad\qquad \leq C_J|J(\omega^*,\alpha^*,y(\omega^*,\alpha^*))-J(\omega^k,\alpha^k,y(\omega^k,\alpha^k))|\\
&\qquad\qquad \leq C_J \bigl( |J(\omega^*,\alpha^*,y(\omega^*,\alpha^*))-J(\omega^k,\alpha^k,y^k)|+\\
&\qquad\qquad\qquad |J(\omega^k,\alpha^k,y^k)-J(\omega^k,\alpha^k,y(\omega^k,\alpha^k))|\bigr)\\
&\qquad\qquad \leq C_J \biggl(\varepsilon^k + | \phi(y^k(t_f))-\phi(y(t_f;\omega^k,\alpha^k))| + \\
&\qquad\qquad\qquad \int_0^{t_f} |\psi(y^k(t),\omega^k(t)) - \psi(y(t;\omega^k,\alpha^k),\omega^k(t))|\,dt\biggr)\\
&\qquad\qquad \leq C_J (1+\eta + t_f \xi)  \varepsilon^k
\end{aligned}
\end{equation}
for a.\,e. $t \in (0,t_f)$, where we added $0=-J(\omega^k,\alpha^k,y^k)+J(\omega^k,\alpha^k,y^k)$, used the triangular inequality
and that $[\omega^k,\alpha^k,y^k]$ is an $\varepsilon$-accurate optimal solution of \eqref{eq:relaxedconvexifiedproblem}.
Moreover, by fixing $t \in [0,t_f]$, adding $0=-y(t;\omega^k,\beta^k)+y(t;\omega^k,\beta^k)$, 
$0=-z(t;\omega^*,\beta^k)+y(t;\omega^*,\beta^k)$ and $0=-y(t;\omega^*,\alpha^k)+y(t;\omega^*,\alpha^k)$ and using the triangular inequality, 
we see that
\begin{equation}\label{eq:estdeltasum}
  \|z^k(t)-y^*(t)\|_X \leq \delta_1(t) + \delta_2(t) + \delta_3(t) + \delta_4(t)
\end{equation}  
with
\begin{equation*}
 \begin{aligned}
  &\delta_1(t)=\|z^k(t)-y(t;\omega^k,\beta^k)\|_X,&&\delta_2(t)=\|y(t;\omega^k,\beta^k)-z(t;\omega^*,\beta^k)\|_X\\
  &\delta_3(t)=\|z(t;\omega^*,\beta^k)-y(t;\omega^*,\alpha^k)\|_X,&&\delta_4(t)=\|y(t;\omega^*,\alpha^k)-y^*(t)\|_X.
 \end{aligned}
\end{equation*}
Observe that $\delta_1(t) \leq \varepsilon^k$, because $z^k$ is an $\varepsilon^k$-accurate solution of \eqref{eq:sysforward}
and thus of \eqref{eq:relaxedconvexifiedproblemB}, \eqref{eq:relaxedconvexifiedproblemC} with controls $\omega^k,\alpha^k$.
By definition of the mild solution and using (H$_1$), we have
\begin{equation}
\begin{aligned}
 \delta_2(t) &\leq \sum_{i=1}^N \int_0^t\|T(t-s)\|_{\Lcal(X)} \|\beta^k_i(s)\| \|(f(s,y(s;\omega^k,\beta^k),\omega^k(s),v^i)\\
 &\qquad\qquad\qquad\qquad- f(s,z(s;\omega^*,\beta^k),\omega^*(s),v^i))\|_X\,ds\\
 &\leq \bar{M}L \int_0^t \delta_2(s) + \|\omega^k(s)-\omega^*(s)\|_U\,ds.
 \end{aligned}
\end{equation}
Using \eqref{eq:omegaconv}, the Gronwall inequality implies that 
\begin{equation}
\delta_2(t) \leq C_J(1+\eta+\xi)e^{\bar{M}Lt}\varepsilon^k.
\end{equation}
From Lemma~\ref{lem:integeralapprox} with $\alpha=\alpha^k$, $\beta=\beta^k$ and $\Delta t = \Delta t^k$ with $\Delta t^k$ from \eqref{eq:deltatdef}, we get that
\begin{equation}
 \max_{i=1,\ldots,N} \left|\int_0^{t} \alpha^k_i(\tau)-\beta_i^k(\tau)\,d\tau\right| \leq (N-1) \Delta t^k,~t\in [0,t_f].
\end{equation}
Moreover, Lemma~\ref{lem:solutionest} used with $y^*=y(\cdot;\omega^*,\alpha^k)$, $\alpha^*=\alpha^k$, $\beta=\beta^k$ and $\varepsilon=(N-1)\Delta t^k$ implies that
\begin{equation}
 \delta_3(t)=\|z(t;\omega^*,\beta^k)-y(t;\omega^*,\alpha^k)\|_X \leq \left((M + Ct)e^{\uM L t}\right)(N-1)\Delta t^k.
\end{equation}
Again by definition of the mild solution we have
\begin{equation}
\begin{aligned}
 \delta_4(t) \leq~&\sum_{i=1}^N\int_0^t \|T(t-s)\|_{\Lcal(X)} \| \alpha^k_i(s) f(s,y(s,\omega^*,\alpha^k),\omega^*(s),v^i)\\
 &\qquad\qquad\qquad\qquad-\alpha_i^*(s) f(s,y(s;\omega^*,\alpha^*),\omega^*(s),v^i)\|_X\,ds.
\end{aligned}
\end{equation}
Adding $0=\alpha^k(s)(-f(s,y(s;\omega^*,\alpha^*),\omega^*(s),v^i)+f(s,y(s;\omega^*,\alpha^*),\omega^*(s),v^i))$ under the 
integral, we obtain
\begin{equation}
 \begin{aligned}
 \delta_4(t) \leq~&\bar{M} \sum_{i=1}^N \int_0^t \|\alpha_i^k(s) [f(s,y(s;\omega^*,\alpha^k),\omega^*(s),v^i)\\
 &\qquad -f(s,y(s,\omega^*,\alpha^*),\omega^*(s),v^i)] \\
 &\qquad + [\alpha_i^k(s)-\alpha^*_i(s)]f(s,y(s;\omega^*,\alpha^*),\omega^*(s),v^i)\|_X\,ds\\
 &\leq \bar{M} L \int_0^t \delta_4(s) + M |\alpha^k_i(s)-\alpha^*_i(s)|\,ds.
 \end{aligned}
\end{equation}
Using again \eqref{eq:omegaconv} and the Gronwall inequality we obtain that 
\begin{equation}
\delta_4(t) \leq M C_J(1+\eta+\xi) e^{\bar{M}Lt}\varepsilon^k.
\end{equation}
Thus, summing up the estimates for $\delta_1(t),\ldots,\delta_4(t)$ and rearranging terms we obtain from \eqref{eq:estdeltasum} that
for all $t \in [0,t_f]$ 
\begin{equation}\label{eq:solest0proof}
 \|z^k(t)-y^*(t)\|_X \leq C_1 \varepsilon^k + C_2 (N-1)\Delta t^k
\end{equation}
with $C_1=(C_J(M+1)(1+\eta+\xi)e^{t_f \bar{M}L}+1)$ and $C_2=(M + t_f C)e^{t_f \uM L}$.
This proves \eqref{eq:solest0}.

By definition of the cost function \eqref{eq:cost} we get from the triangular inequality that
\begin{equation}
\begin{aligned}
 &|J(\omega^*,\alpha^*,y^*) - J(u^k,v^k,z^k)| \leq |\phi(y^*(t))-\phi(z^k(t))|\\
 &\qquad\qquad\qquad\qquad+ \int_0^{t_f} | \psi(y^*(t),\omega^*(t)) - \psi(z^k(t),\omega^k(t))|\,dt
\end{aligned}
\end{equation}
so that using hypothesis (H$_1$), \eqref{eq:omegaconv} and \eqref{eq:solest0proof} we obtain
\begin{equation}\label{eq:Jest0proof}
 |J(\omega^*,\alpha^*,y^*) - J(u^k,v^k,z^k)| \leq C_3 \varepsilon^k + C_4 (N-1)\Delta t^k
\end{equation}
with the constants $C_3=(\eta + t_f \xi) C_1 + t_f \xi C_J$ and $C_4=(\eta+t_f\xi)C_2(t_f)$. This proves \eqref{eq:Jest0}.

Next, suppose that the main loop in Algorithm~\ref{algo:iter} terminates in step {\footnotesize $4$:} or in step {\footnotesize $7$:}. 
In the first case, the termination criterion implies that
\begin{equation}
 |J(u^*,v^*,z^*)-J(\omega^*,\alpha^*,y^*)|=|J(\omega^k,\alpha^k,y^k)-J(\omega^*,\alpha^*,y^*)| \leq \varepsilon^k \leq \varepsilon
\end{equation}
for some $k$, because $J(\omega^k,\alpha^k,z^k)$ is an $\varepsilon^k$-accurate optimal solution of \eqref{eq:relaxedconvexifiedproblem}. 
Similarly, in the second case the termination criterion implies that
\begin{equation}
\begin{aligned}
 |J(u^*,v^*,z^*)-J(\omega^*,\alpha^*,y^*)| \leq ~&|J(u^*,v^*,z^*)-J(\omega^k,\alpha^k,y^k)|\\
 &\qquad +|J(\omega^k,\alpha^k,y^k)-J(\omega^*,\alpha^*,y^*)|\\ 
 \leq~& |J^k - J^k_\rel| + \varepsilon^k \leq \frac{\varepsilon}{2} + \frac{\varepsilon}{2} \leq \varepsilon
 \end{aligned}
\end{equation}
for some $k$. This proves \eqref{eq:MIOCPoptValueEst} under the assumption that Algorithm~\ref{algo:iter} terminates.

Finally suppose that Algorithm~\ref{algo:iter} loops infinitely many times, that is,  
\begin{equation}\label{eq:contradiction}
 |J^k_\rel-J^k| > \frac{\varepsilon}{2}~\text{or}~\varepsilon^k>\frac{\varepsilon}{2}~\text{for all}~k=0,1,2,\ldots
\end{equation}
By adding $0=-y^*(t)+y^*(t)$ and using that $y^k$ is an $\varepsilon^k$-accurate optimal solution, we obtain from 
\eqref{eq:solest0proof} that
\begin{equation}\label{eq:zkykcompare}
\begin{aligned}
 \|z^k(t) - y^k(t)\|_X \leq &\|z^k(t)-y^*(t)\|_X + \|y^*(t)-y^k(t)\|_X \\
 & \leq C_1\varepsilon^k + C_2(N-1)\Delta t^k + \varepsilon^k.
\end{aligned} 
\end{equation}
Using that $\varepsilon^k \to 0$ and $\Delta t^k \to 0$ as $k \to \infty$ by assumption, we see from \eqref{eq:zkykcompare} 
that $\sup_{t \in [0,t_f]} \|z^k(t) - y^k(t)\|_X \to 0$ as $k \to \infty$.
By definition of $J^k_\rel$ and $J^k$ and using the triangular inequality we have
\begin{equation} \label{eq:JkrelJkest}
\begin{aligned}
 |J^k_\rel-J^k| =~&|J(\omega^k,\alpha^k,y^k)-J(\omega^k,\beta^k,z^k)| \leq |\phi(y^k(t_f))-\phi(z^k(t_f))| \\
 &\qquad + \int_0^{t_f} | \psi(y^k(t),\omega^k(t))-\psi(z^k(t),\omega^k(t))|\,dt,
\end{aligned}
\end{equation}
so that $\sup_{t \in [0,t_f]} \|z^k(t) - y^k(t)\|_X \to 0$ as $k \to \infty$ implies by continuity of $\phi$ and $\psi$ that also
$|J^k_\rel-J^k| \to 0$ as $k \to \infty$. Together with the assumption that 
$\varepsilon^k \to 0$ as $k \to \infty$ this contradicts \eqref{eq:contradiction} and completes the proof. \qed
\end{proof}

Theorem~\ref{thm:convergence} can be seen as a performance analysis of the relaxation method proposed in Algorithm~\ref{algo:iter}. 
The estimates \eqref{eq:solest0} and \eqref{eq:Jest0} prove a bilinear dependency of the mixed-integer control approximation 
error for the differential state and the optimal value in terms of the chosen maximal integer-control 
discretization mesh size $\Delta t^k$ and accuracy $\varepsilon^k$. This relates to the convergence 
speed of Algorithm~\ref{algo:iter} in terms of the chosen refinements for $\Delta t^k$ and $\varepsilon^k$.
Note that the estimates \eqref{eq:solest0} and \eqref{eq:Jest0} suggest to choose $\Delta t^k$ and $\varepsilon^k$ 
of the same order. On the other hand, Theorem~\ref{thm:convergence} can be regarded as an existence result of
suboptimal solutions for (\MIP). To emphasize this, we formulate the precise statement explicitly.

\begin{corollary}\label{cor:convergence} Under the hypothesis {\normalfont (H$_0$)}--{\normalfont (H$_3$)} there exists for every $\varepsilon>0$ a feasible
solution $(u^{\varepsilon},v^{\varepsilon},z^{\varepsilon}) \in U_{[0,t_f]} \times V_{[0,t_f]} \times X_{[0,t_f]}$
of problem (\MIP) satisfying
\begin{equation}\label{eq:Jest1}
  J(u^{\varepsilon},v^{\varepsilon},z^{\varepsilon}) \leq J(\omega^*,\alpha^*,y^*) + \varepsilon,
\end{equation}
where $(\omega^*,\alpha^*,y^*)$ is the optimal solution of the relaxed problem~\eqref{eq:relaxedconvexifiedproblem}.
\end{corollary}
\begin{proof}
 Apply Theorem~\ref{thm:convergence} with $\varepsilon^k=0$ for all $k$. Then, $(\omega^k,\alpha^k,y^k)=(\omega^*,\alpha^*,y^*)$
 and it can be seen from proof of Theorem~\ref{thm:convergence} 
 that (H$_4$) is then not needed for the estimate \eqref{eq:Jest0}. Moreover, $\Xi_0$ is contained in the feasible set of problem \eqref{eq:relaxedconvexifiedproblem}
 and thus $J(\omega^*,\alpha^*,y^*) \leq J(u,v,z)$ for all $(u,v,z) \in \Xi_0$. Thus, \eqref{eq:Jest1} follows from \eqref{eq:MIOCPoptValueEst}. \qed
\end{proof}

Hypothesis~(H$_0$) can be checked by classical arguments, cf., e.\,g., \cite{Troeltzsch2010}. Hypothesis (H$_1$)--(H$_3$) are needed to 
estimate the proximity of mixed-integer solutions of (MIOCP) to the optimal solution of 
the relaxed problem \eqref{eq:relaxedconvexifiedproblem} while hypothesis (H$_4$) guarantees in a sense the proximity of 
the $\varepsilon$-accurate optimal solutions of \eqref{eq:relaxedconvexifiedproblem} to the optimal ones.
Hypothesis~(H$_1$) and (H$_3$) are standard assumptions and can be weakend to appropriate `local' versions using standard arguments. This
can also tighten the estimates in Theorem~\ref{thm:convergence}. The last conclusion of Theorem~\ref{thm:convergence} concerning the termination of 
Algorithm~\ref{algo:iter} and Corollary~\ref{cor:convergence} 
even hold for the functions $\phi$ and $\psi$ just continuous as it can be seen from the respective proofs. 
Hypothesis (H$_4$) can be verified for problems that are well-posed in the Tikhonov sense, cf. the discussion in 
\cite[Section~4]{LeugeringKogut2011}. For sufficiently regularized parabolic problems it could also be checked 
using methods as in \cite{MalanowskiTroeltzsch1999}. Alternatively, instead of invoking Theorem~\ref{thm:convergence}, 
Corollary~\ref{cor:convergence} or a weaker conclusion presented in Proposition~\ref{prop:convergenceWeaker} below may be used where 
(H$_4$) is not needed. 

Hypothesis {\normalfont (H$_2$)} of Theorem~\ref{thm:convergence} clearly imposes certain regularity
assumptions on the linear operator $A$ generating the semigroup $\{T(t)\}_{t \geq 0}$, the
function $f$, but also on the time regularities of the optimal control functions of the relaxed problem \eqref{eq:relaxedconvexifiedproblem} 
in $U_{[0,t_f]}$ and $\tilde{V}_{[0,t_f]}$. The main difficulty with proving (H$_2$) is that $y^*$ as a solution of \eqref{eq:vc} with $A$ unbounded may 
only be continuous and not absolutely continuous in time, hence not necessarily differentiable almost 
everywhere as this is always true when $A=0$ (with $T(\cdot)=\Id$) and $X$ is a finite dimensional space. This can be delicate in particular for
nonlinear systems. We will therefore exemplary discuss hypothesis {\normalfont (H$_2$)} in Example~\ref{ex:lotkadiff2d} below for the case of a semilinear 
system where $A$ is the generator of an analytic semigroup. A more general analysis is possible for linear systems 
\begin{equation}\label{eq:linearsysuc}
 \dot{z}(t)=Az(t)+f(t,u(t),v(t)),\quad z(0)=z_0
\end{equation}
when $f$ is sufficiently smooth. We formulate this as an auxiliary result.

\begin{proposition}\label{prop:OnH2Auniformlybounded}
Consider the problem (\MIP) with equation \eqref{eq:sysuc} replaced by equation \eqref{eq:linearsysuc},
let $\bar{M}=\sup_{t \in [0,t_f]}\|T(t)\|_{\Lcal(X)}$ and suppose that the functions 
$g^i\:[0,t_f] \to X$ defined by $g^i(t)=f(t,\omega^*(t),v^i)$, $i=1,\ldots,N$,
satisfy the following conditions.
\begin{itemize}
 \item[(i)] $g^i(t) \in D(A)$ for a.\,e. $t \in [0,t_f]$ and there exists constants $\bar{L}^i$
such that
\begin{equation*}
\esssup_{t \in [0,t_f]}\left\|A g^i(t)\right\|_X \leq \bar{L}^i.
\end{equation*}
 \item[(ii)] $g^i$ is differentiable for a.\,e. $t \in [0,t_f]$ and there exist constants $\bar{C}^i$
such that 
\begin{equation*}
\esssup_{t \in [0,t_f]}\left\|\frac{d}{dt}g^i(t)\right\|_X \leq \bar{C}^i.
\end{equation*}
\end{itemize}
Then hypothesis {\normalfont (H$_2$)} of Theorem~\ref{thm:convergence} holds with $C_i:=\bar{M}\left(\bar{C}^i+\bar{L}^i \right)$.
\end{proposition}
\begin{proof}
From condition (i) we get from the chain rule that
\begin{equation*}
\frac{d}{ds} T(t-s)g(t) = T(t-s)\frac{d}{ds}g(t)-T(t-s)Ag(t)
\end{equation*}
for all $i=1,\ldots,N$ and thus, by taking the norm, applying the triangular inequality and using
the definition of the constants in (i) and (ii) we obtain
\begin{equation*}
 \left\|\frac{d}{ds} T(t-s)f(s,\omega^*(s),v^i)\right\|_X \leq \bar{M}\left(\bar{C}^i+\bar{L}^i \right).
\end{equation*} 
\qed
\end{proof}
The conditions (i) and (ii) are a natural extension of the differentiability assumptions imposed in 
\cite[Corollary~6]{SagerBockDiehl2011} for the case when $A=0$ and $X=\RR^n$. In Section~\ref{sec:examples}, 
we will use such arguments in order to verify hypothesis (H$_2$) in Example~\ref{ex:heatrelax2d}.

We note that the relaxation method works under much weaker assumptions with slightly weaker conclusions.
Suppose that we replace the main hypothesis {\normalfont (H$_0$)} by the following much weaker hypothesis.
\vspace*{0.5em}
\begin{list}{\labelitemi}{\itemsep=1em\leftmargin=2em}
 \item[{\normalfont (H$_0'$)}] Problem \eqref{eq:relaxedconvexifiedproblem} has a feasible
  solution in $U_{[0,t_f]} \times \tilde{V}_{[0,t_f]} \times X_{[0,t_f]}$.
\end{list}
\vspace*{0.5em}
Then, we still get the following result, being useful in particular in many practical 
applications when the solutions $[\omega^k,\alpha^k,y^k]$ found in step
{\footnotesize $3$:} only satisfy, for example, necessary optimality conditions (up to an accuracy of
$\varepsilon^k$). The following conclusion can then still be very useful in order to provide
bounds for the mixed-integer problem (\MIP) and we will take advantage of this
when discussing the examples in Section~\ref{sec:examples}. Nevertheless, an approximation of
a globally optimal solution of problem (\MIP) can of course only be obtained
when the relaxed problem in step {\footnotesize $3$:} is solved to
$\varepsilon^k$-global optimality.

\begin{proposition}\label{prop:convergenceWeaker} {\nopagebreak Under the
hypothesis {\normalfont (H$_0'$)}, consider Algorithm~\ref{algo:iter} where we replace 
step {\footnotesize $3$:} by
\begin{list}{\labelitemi}{\itemsep=1em\leftmargin=1.5em}
 \item[{\footnotesize $3'$:}] \quad Select some $[\omega^k,\alpha^k,y^k]$ such that $\omega^k$ and $\alpha^k$
 is feasible for problem \eqref{eq:relaxedconvexifiedproblem} and $y^k$ is an $\varepsilon^k$-accurate solution 
 of \eqref{eq:relaxedconvexifiedproblemB}, \eqref{eq:relaxedconvexifiedproblemC}.
\end{list}}
Assume hypothesis {\normalfont (H$_1$)} and that the hypothesis {\normalfont (H$_2$)} and {\normalfont (H$_3$)} hold 
with $\omega^*=\omega^k$ and $y^*=y(\cdot;\omega^k,\alpha^k)$ with constants $C^k$ and $M^k$ for all $k=0,1,2,\ldots$
and $y(\cdot;\omega^k,\alpha^k)$ being the mild solution of \eqref{eq:relaxedconvexifiedproblemB}, \eqref{eq:relaxedconvexifiedproblemC}. 
Further assume that, as in Theorem~\ref{thm:convergence}, $\varepsilon^k \to 0$ and that the sequence $\{\Gcal^{k}\}_k$ 
is such that $\Delta t^k \to 0$. Define the constants $C_1 = 2$, $C_2^k=\left((M^k + t_f C^k)e^{t_f \uM L}\right)$, 
$C_3=(\eta + t_f \xi)C_1$ and $C_4^k=(\eta +t_f \xi) C^k_2$, where $\uM=\sup_{t \in [0,t_f]} \|T(t)\|_{\Lcal(X)}$ 
and the constants $\eta$, $\xi$ and $L$ are given by hypothesis {\normalfont (H$_1$)}. 
Then, $[u^k,v^k,z^k]$ defined by Algorithm~\ref{algo:iter} is in $\Xi_{\varepsilon^k}$ for all $k=0,1,2,\ldots$ and 
satisfies the estimates
\begin{equation}\label{eq:solest0prop}
 \|y^k(t)-z^k(t)\|_X \leq C_1 \varepsilon^k + C_2^k(N-1) \Delta t^k,\quad t \in [0,t_f],
\end{equation}
and 
\begin{equation}\label{eq:Jest0prop}
\begin{aligned}
 &|J(\omega^k,\alpha^k,y^k) - J(u^k,v^k,z^k)| \leq C_3 \varepsilon^k + C_4^k (N-1)\Delta t^k.
 \end{aligned}
\end{equation}
In particular, if $M^k$ and $C^k$ can be chosen independently of $k$, then Algorithm~\ref{algo:iter} terminates in a 
finite number of steps $\kappa$ with an $\varepsilon$-feasible mixed-integer solution $[u^*,v^*,z^*] \in \Xi_{\varepsilon}$ 
of Problem (\MIP) satisfying the estimate 
\begin{equation}\label{eq:MIOCPoptValueEstprop}
|J(\omega^\kappa,\alpha^\kappa,y^\kappa) - J(u^*,v^*,z^*)| \leq \varepsilon,
\end{equation}
where $\varepsilon>0$ was chosen arbitrarily in step {\footnotesize\normalfont $1$:}.
\end{proposition}

\begin{proof}
Let the assumptions of Proposition~\ref{prop:convergenceWeaker} hold true. First observe that the sequence
$[u^k,v^k,z^k]$ obtained by Algorithm~\ref{algo:iter} with step {\footnotesize $3$:} replaced by
{\footnotesize $3'$:} is $\varepsilon^k$-feasible for the problem (\MIP) by the same arguments as in the proof
of Theorem~\ref{thm:convergence}. 

To show \eqref{eq:solest0prop} and \eqref{eq:Jest0prop}, let $y(\cdot;\omega^k,\alpha^k)$ denote the mild solution 
of \eqref{eq:relaxedconvexifiedproblemB}, \eqref{eq:relaxedconvexifiedproblemC} and $z(\cdot;u^k,\beta^k)$ be the mild solution 
of \eqref{eq:sysforward} with the respective controls. Then,
\begin{equation}
 \|y^k(t)-z^k(t)\|_X \leq \|y(t;\omega^k,\alpha^k) - z(t;u^k,\beta^k)\|_X + 2\varepsilon^k,~t \in [0,t_f]
\end{equation}
because $y^k$ and $z^k$ are both $\varepsilon^k$-accurate solutions of \eqref{eq:relaxedconvexifiedproblemB}, \eqref{eq:relaxedconvexifiedproblemC}
and \eqref{eq:sysforward}, respectively. From Lemma~\ref{lem:integeralapprox} with $\alpha=\alpha^k$, $\beta=\beta^k$, 
and $\Delta t = \Delta t^k$ with $\Delta t^k$ from \eqref{eq:deltatdef}, we get that
\begin{equation}
 \max_{i=1,\ldots,N} \left|\int_0^{t} \alpha^k_i(\tau)-\beta_i^k(\tau)\,d\tau\right| \leq (N-1) \Delta t^k,~t\in [0,t_f].
\end{equation}
Moreover, under hypothesis {\normalfont (H$_1$)} and the assumption that the hypothesis {\normalfont (H$_2$)} and {\normalfont (H$_3$)} hold 
with $\omega^*=\omega^k$ and $y^*=y^k$ with constants $C^k$ and $M^k$ for all $k=0,1,2,\ldots$, we may apply Lemma~\ref{lem:solutionest} with 
$y^*=y(\cdot;\omega^k,\alpha^k)$, $\alpha^*=\alpha^k$, $\beta=\beta^k$, $\omega^*=\omega^k=u^k$ and $\varepsilon=(N-1)\Delta t^k$ and obtain that
\begin{equation*}
\|y(t;\omega^k,\alpha^k) - z(t;u^k,\beta^k)\|_X \leq \left((M^k + C^k t)e^{\uM L t}\right)(N-1)\Delta t^k,~t\in [0,t_f].
\end{equation*}
This proves \eqref{eq:solest0prop} with $C_1 = 2$ and $C_2^k=\left((M^k + t_f C^k)e^{t_f \uM L}\right)$. Using again the continuity 
assumptions of $\phi$ and $\psi$ in {\normalfont (H$_1$)}, we obtain similarly as in the proof of Theorem~\ref{thm:convergence} that
\begin{equation}
 \|J(\omega^k,\alpha^k,y^k)-J(u^k,v^k,z^k)\|_X \leq C_3 \varepsilon^k + C_4^k(t)(N-1)\Delta t^k 
\end{equation}
$C_3 = C_1 (\eta + t_f \xi)$ and $C_4^k=(\eta + t_f \xi)C_2^k$. This proves \eqref{eq:Jest0prop}.

Finally, when the constants $C^k$ and $M^k$ can be chosen independently of $k$, $\varepsilon^k \to 0$ and $\Delta t^k \to 0$, then 
$C_4^k$ is independent of $k$ and we get by similar arguments as in the proof of Theorem~\ref{thm:convergence} that the iteration terminates 
after a finite number of steps such that \eqref{eq:MIOCPoptValueEstprop} holds. \qed
\end{proof}

In order to solve the optimal control problem~\eqref{eq:relaxedconvexifiedproblem} 
numerically, the problem may have to be (adaptively) discretized. In particular, direct or indirect 
numerical methods may be used. For an introduction to the basic concepts see, e.\,g., 
\cite{HinzePinnauUlbrichUlbrich2009}. Depending on the method of choice for the time 
discretization of the control functions $\omega \in U_{[0,t_f]}$ and $\alpha \in \tilde{V}_{[0,t_f]}$
it may in many cases be advantageous to discretize $\omega^k$ and $\alpha^k$ simultaneously using the 
grid $\Gcal^k$. This is for example implemented in the software package MS MINTOC 
designed for solving mixed-integer optimal control problems with ordinary differential 
equations \cite{SagerBockReinelt2009,Sager2009}. 

We conclude this section with an interesting remark saying that, in the fashion of Theorem~\ref{thm:convergence} 
and Proposition~\ref{prop:convergenceWeaker}, the relaxation method can also deal with state constraints.

\begin{remark}\label{rem:stateconstraints}
Suppose that we wish to include a constraint of the form
\begin{equation}\label{eq:stateconstraint}
 G(z(t),t)\geq 0,\quad t \in [0,t_f]
\end{equation}
in the mixed-integer optimal control problem (\MIP). Including this constraint also in 
\eqref{eq:relaxedconvexifiedproblem} with $z(t)$ replaced by $y(t)$ and assuming 
that there exists a function $\zeta \in L^{\infty}(0,t_f)$ such that
\begin{equation}\label{eq:GLip0}
|G(y_1,t)-G(y_2,t)| \leq \zeta(t)\|y_1 - y_2\|_X,\quad y_1,y_2 \in X
\end{equation} 
then \eqref{eq:solest0} yields that
\begin{equation}\label{eq:stateconstraintest}
 |G(y^*(t),t)-G(z^k(t),t)| \leq \zeta(t)C_1 \varepsilon^k + \zeta(t)C_2(N-1)\Delta t^k
\end{equation}
with $C_1$ and $C_2$ as in Theorem~\ref{thm:convergence}. This shows also a bilinear dependency 
of the integer-control approximation error for the state constraint violation on $\Delta t^k$ and $\varepsilon^k$. 
The conclusion of Proposition~\ref{prop:convergenceWeaker} can be adapted
accordingly.
\end{remark}

\section{Combinatorial Constraints}\label{sec:cconstraints}

\begin{floatingalg}
\rule{\textwidth}{0.4pt}
\begin{algorithm}\label{algo:iter2}\mbox{} \normalfont
Consider Algorithm~\ref{algo:iter} where we replace steps {\footnotesize $4$:} {\footnotesize $5$:} and {\footnotesize $7$:} by
\begin{list}{\labelitemi}{\leftmargin=2em}
\item[{\footnotesize $4'$:}] If $\alpha^k \in \PC(0,t_f;\{0,1\}^N)$, \eqref{eq:cconstraints} holds with $v=\sum_{i=1}^N \alpha_i^k v^i$ and 
$\varepsilon^k \leq \varepsilon$, then set $\beta^k(t)=\alpha^k(t)$ and $z^k(t)=y^k(t)$ for $t \in [0,t_f]$ and STOP.
\item[{\footnotesize $5'$:}] 
Using $\Gcal^k$, define a piecewise constant function
            $\beta^k=(\beta^k_1,\ldots,\beta^k_N)\: [0,t_f] \to \{0,1\}^N$ by
            \begin{equation}\label{eq:omegadef2}\tag{A$_4$}
            \beta^k_i(t) = p^{k,*}_{i,j},\quad t\in [t^k_j,t^k_{j+1})
            \end{equation}
            where $p^{k,*}_{i,j}$ is given by the solution of the min-max
            problem
            \begin{equation}\label{eq:minmax}\tag{A$_5$}
            \left\{\begin{aligned}
            & \min_{p^k}~J_{\mathrm{sub}}(p^k)=\max_{i=1,\ldots,N}
            \max_{r=1,\ldots,n^k} \left| \sum_{l=1}^r (q_{i,l}^k-p_{i,l}^k)\Delta t^k_l\right|\\
            & \text{subject to}\\
            & \quad \sum_{r=1}^{n^k} |p^k_{i,r}-p^k_{j,{r+1}}| \leq
            K^{i,j},\quad i \in I,~j\in J\\
            & \quad \sum_{i=1}^N p_{i,r}^k = 1,\quad r=1,\ldots,n^k\\ 
            & \quad p_{i,r}^k \in \{0,1\},\quad i=1,\ldots,N,~r=1,\ldots,n^k
            \end{aligned}\right.
            \end{equation}
            with $\Delta t^k_l=t^k_{l+1}-t^k_l$, $l=1,\ldots,n^k$, and
            \begin{equation}\label{eq:qdef}\tag{A$_6$}
            q_{i,l}^k=\frac{1}{\Delta t_l^k} \int_{t_l}^{t_{l+1}}
            \alpha^k_i(t)\,dt,\quad i=1,\ldots,N,~l=1,\ldots,n^k.
            \end{equation}
   \item[{\footnotesize $7'$:}] If $|J_{\mathrm{sub}}(p^{k,*})-J_{\mathrm{sub}}(p^{k-1,*})|<\varepsilon$ or
   $k \geq k_{\max}$ then STOP.
\end{list}
\vspace*{0.5em}
\end{algorithm}
\rule{\textwidth}{0.4pt}
\end{floatingalg}

Suppose we wish to include combinatorial constraints of the form
\begin{equation}\label{eq:cconstraints}
\#_{v^i \cvto v^j}(v) \leq K^{i,j},\quad i\in I,~j\in J
\end{equation}
into the mixed-integer optimal control problem (\MIP) given by
\eqref{eq:cost}--\eqref{eq:constraintuc}, where $\#_{v^i \cvto v^j}(v)$ 
denotes the number of switches of the control function $v\: [0,t_f] \to \Vad$
from value $v^i$ to value $v^j$, $K^{i,j}$ are given, non-negative 
constants and $I,J \subset \{1,\ldots,N\}$.  

Note that the relaxation method considered in Section~\ref{sec:relax} typically
satisfies
\begin{equation*}
\#_{v^i \cvto v^j}(v) \to +\infty 
\end{equation*}
for some $i,j \in \{1,\ldots,N\}$ as we let $\varepsilon \to 0$, so eventually
violating \eqref{eq:cconstraints} for small $\varepsilon$.
Therefore, along the lines of \cite{SagerJungKirches2011}, we propose in Algorithm~\ref{algo:iter2} 
a modification of Algorithm~\ref{algo:iter}. The min-max problem \eqref{eq:minmax} can be written as a standard mixed-integer
linear problem (MILP) using slack variables and can be computed efficiently
\cite{SagerJungKirches2011}. We then have the following result.

\begin{theorem}\label{thm:cconstaints} Suppose that the hypotheses of Theorem~\ref{thm:convergence}
hold true and let $C_1$, $C_2$, $C_3$ and $C_4$ be as in Theorem~\ref{thm:convergence}. Then 
Algorithm~\ref{algo:iter2} terminates for every $\varepsilon>0$ and $k_{\max} \geq 0$ after a finite number of 
steps $\kappa \leq k_{\max}$ with an $\varepsilon^{\kappa}$-feasible solution 
$[u^*,v^*,z^*] \in \Xi_{\varepsilon^{\kappa}}$ of problem (\MIP) satisfying the combinatorial constraints \eqref{eq:cconstraints} 
and the estimates 
\begin{equation}\label{eq:solest2}
\|y^*(t)-z^*(t)\|_X \leq C_1\varepsilon^{\kappa} + C_2 (J_{\mathrm{sub}}(p^{\kappa,*})+\delta),\quad t \in [0,t_f],
\end{equation}
and
\begin{equation}\label{eq:Jest2}
\begin{aligned} 
 &|J(\omega^*,\alpha^*,y^*)-J(u^*,v^*,z^*)| \leq C_3 \varepsilon^{\kappa} + C_4 (J_{\mathrm{sub}}(p^{{\kappa},*})+\delta)
\end{aligned}
 \end{equation}
 for $J_{\mathrm{sub}}(p^{{\kappa},*})$ given by Algorithm~\ref{algo:iter2} and some $0 \leq \delta \leq \max_{l=1,\ldots,n^{\kappa}} \Delta t_l^{\kappa}$. 
\end{theorem}
\begin{proof}\mbox{}
Algorithm~\ref{algo:iter2} terminates by the criterion in step {\footnotesize $7'$:} after $\kappa$ steps, $\kappa \leq k_{\max}$, with a solution $[u^*,v^*,z^*]$.
We first show that this solution is $\varepsilon^\kappa$-feasible for the problem (\MIP). We have $u^*=\omega^\kappa \in \Uad$ and by definition of $\omega^\kappa$ in 
step {\footnotesize $3$:}, $v^\kappa(t) = \sum_{i=1}^N \beta_i^\kappa(t) v^i$, $t \in [0,t_f]$, by definition in step {\footnotesize $6$:},
so $v^\kappa(t) \in \Vad$, $t \in [0,t_f]$, because $\beta_i^k(t) \in \{0,1\}$ for all $i,k$ and $t \in [0,t_f]$ 
as seen from \eqref{eq:omegadef2} and the constraints in \eqref{eq:minmax}.By construction in step 
{\footnotesize $6$:}, $z^\kappa$ is an $\varepsilon^\kappa$-accurate solution of \eqref{eq:sysforward}. Thus, recalling 
\eqref{eq:Xidef}, $[u^*,v^*,z^*] \in \Xi_{\varepsilon^{\kappa}}$. The constraints in \eqref{eq:minmax} also ensure that the combinatorial 
constraints \eqref{eq:cconstraints} are satisfied. 

Next we show \eqref{eq:solest2}. For all $k=0,1,\ldots$ the cost function in \eqref{eq:cconstraints} is defined as
\begin{equation}
J_{\mathrm{sub}}(p^{k,*})=\max_{i=1,\ldots,N} \max_{r=1,\ldots,n^k} \left| \sum_{l=1}^r (q_{i,l}^k-p_{i,l}^k) \Delta t_l^k \right|.
\end{equation}
By definition of $q_{i,l}^k$ in \eqref{eq:qdef} and $\beta_i^k(t)$ in \eqref{eq:omegadef2}
and rearranging terms, we get
\begin{equation}
J_{\mathrm{sub}}(p^{k,*})=\max_{i=1,\ldots,N} \max_{r=1,\ldots,n^k} \left| 
\int_0^{t_{r+1}}\alpha_i^k(t)-\beta_i^k(t)\,dt \right|.
\end{equation}
Using that $\alpha_i^k(t) \in [0,1]$ and $\beta_i^k(t) \in \{0,1\}$ for all $t \in [0,t_f]$,
this yields
\begin{equation}
J_{\mathrm{sub}}(p^{k,*})=\max_{i=1,\ldots,N} \sup_{t \in [0,t_f]} \left| 
\int_0^{t}\alpha_i^k(\tau)-\beta_i^k(\tau)\,d\tau \right| - \delta
\end{equation}
for some $0 \leq \delta \leq \max_{l=1,\ldots,n^k} \Delta t_l^k$.
Fixing some $t \in [0,t_f]$, we have as in the proof of Theorem~\ref{thm:convergence}
\begin{equation}
 \|z^k(t)-y^*(t)\|_X \leq \delta_1(t)+\delta_2(t)+\delta_3(t)+\delta_4(t),
\end{equation}
with $\delta_i(t)$, $i=1,\ldots,4$, as in \eqref{eq:deltatdef}. Moreover, as in the proof
of Theorem~\ref{thm:convergence}, we see that
$\delta_1(t) \leq \varepsilon^k$, $\delta_2(t) \leq C_J(1+\eta+\xi)e^{\bar{M}Lt}\varepsilon^k$ and 
$\delta_4(t)\leq M C_J(1+\eta+\xi) e^{\bar{M}Lt}\varepsilon^k$. Using {\normalfont (H$_1$)}--{\normalfont (H$_3$)},
we can apply Lemma~\ref{lem:solutionest} with $y^*=y(\cdot;\omega^*,\alpha^k)$, $\alpha^*=\alpha^k$, $\beta=\beta^k$ and $\varepsilon=J_{\mathrm{sub}}(p^{k,*})+\delta$ 
and get that
\begin{equation}
\delta_3(t)=\|y(t;\omega^*,\beta^k)-y(t;\omega^*,\alpha^k)\|_X \leq (M + C t)e^{\uM L t} (J_{\mathrm{sub}}(p^{k,*})+\delta).
\end{equation} 
Summing up the estimates for $\delta_1(t),\ldots,\delta_4(t)$ and rearranging terms, this
proves \eqref{eq:solest2} with the definition of $C_1$ and $C_2$ as in Theorem~\ref{thm:convergence}. 
The estimate \eqref{eq:Jest2} then follows from \eqref{eq:solest2} and the definition of the constants
$C_3$ and $C_4$ as in Theorem~\ref{thm:convergence} using the definition of 
the cost function $J$ in \eqref{eq:cost} and the Lipschitz constants $\eta$ and $\xi$ from {\normalfont (H$_1$)}. 
This completes the proof of Theorem~\ref{thm:cconstaints}. \qed
\end{proof}

\begin{remark} As already remarked in the case without combinatorial constraints, the method can 
also deal with state constraints such as \eqref{eq:stateconstraint}. Assuming again existence of 
a function $\zeta \in L^{\infty}(0,t_f)$ such that \eqref{eq:GLip0} holds true, \eqref{eq:solest2} 
yields a bounded deviation of the feasible reference trajectory
\begin{equation}\label{eq:stateconstraintest2}
 |G(y^*(t),t)-G(z^k(t),t)| \leq \zeta(t)C_1\varepsilon^k + \zeta(t)C_2(J_{\mathrm{sub}}(p^{k,*})+\delta),
\end{equation}
and hence a bound on the worst case constraint violation. Also, the conclusion of 
Proposition~\ref{prop:convergenceWeaker} can be adapted analogously.
\end{remark}

\section{Examples}\label{sec:examples}
In this section we discuss the hypothesis (H$_1$)--(H$_3$) of Theorem~\ref{thm:convergence}
exemplary for a linear and a semilinear control problem where $A$ is the generator of an 
analytic semigroup in view of Proposition~\ref{prop:convergenceWeaker} and present numerical 
results for a test problem in each case using the conclusions.

\subsection{A linear parabolic equation with lumped controls}
Let $\Omega$ be a domain in $\RR^n$ and $f_i\: \Omega \to \RR$,
$i=1,\ldots,N$, be fixed control profiles. Consider the internally 
controlled heat equation
\begin{equation}\label{eq:heatlumped}
\left\{\begin{aligned}
&\frac{\partial z}{\partial t}(x,t) - \rho \sum_{j=1}^n \frac{\partial^2
z}{\partial x^2_j}(x,t) = f_{\sigma(t)}(x)u(t),\quad \text{in}~Q\\ 
&z(x,t) = 0,\quad \text{on}~\Sigma\\
&z(x,0)=z_0(x),\quad \text{in}~\Omega
\end{aligned}\right.
\end{equation}
where $Q=\Omega \times (0,t_f)$, $\Sigma=\partial\Omega \times (0,t_f)$ and
$\rho$ is a positive constant.

Suppose that for some $\lambda_1 \geq 0$ and $\lambda_2>0$ the control task is to minimize 
the cost function
\begin{equation}\label{eq:heatcost}
J = \int_\Omega |z(t_f,x)|^2\,dx + \lambda_1 \int_0^{t_f} \int_\Omega
|z(t,x)|^2\,dx\,dt + \lambda_2 \int_0^{t_f} |u(t)|^2\,dt 
\end{equation} 
where $z$ is the weak solution of \eqref{eq:heatlumped} by selecting $u\:[0,t_f] \to \RR$ 
and a switching signal $\sigma(\cdot)\: [0,t_f] \to \{1,\ldots,N\}$
determining the control profile $f_i$ applied at time $t \in [0,t_f]$. 

In order to write the above problem in abstract form \eqref{eq:sysuc}, 
we let $X=L^2(\Omega)$, set $V=U=\Uad=\RR$, $\Vad=\{1,\ldots,N\}$ and
define $f\: [0,t_f] \times U \times V \to X$ by 
$f(t,u,v)(x):=f_v(x)u$, $\phi(z)=\|z\|^2_X$, 
$\psi(z,u)=\lambda_1\|z\|^2_X + \lambda_2|u|^2$ and define $(A,D(A))$ as
\begin{equation}
\begin{aligned}
&D(A) = H^2(\Omega) \cap H^{1}_0(\Omega)\\
&(Az)(x) = \sum_{j=1}^n \frac{\partial^2 z}{\partial
x^2_j} (x),~z \in D(A).
\end{aligned}
\end{equation}
It is well-known that $(A,D(A))$ is the generator of a strongly continuous
(analytic) semigroup of contractions $\{T(t)\}_{t \geq 0}$ on $X$, see, 
e.\,g., \cite{Pazy1983}. We choose $X_{[0,t_f]}=C([0,t_f];X)$, 
$U_{[0,t_f]}=PC(0,t_f;\RR)$ and $V_{[0,t_f]}=L^\infty(0,t_f;\RR)$.

Let $[\omega^k,\alpha^k,y^k]$ be a sequence of feasible solutions of the corresponding
relaxed problem \eqref{eq:relaxedconvexifiedproblem} so that $[\omega^k,\alpha^k,y^k] \in \Scal$ 
for $k=0,1,2,\ldots$, where $\Scal$ is a bounded subset of $X_{[0,t_f]} \times U_{[0,t_f]} \times \tilde{V}_{[0,t_f]}$ 
and assume that the fixed control profiles $f_i$ satisfy
\begin{equation}\label{eq:ShmoothProfiles}
f_i \in D(A)~\text{for all}~i=1,\ldots,N.
\end{equation}

We now want to check if the assumptions of Proposition~\ref{prop:convergenceWeaker} are satisfied. 
For this, we can restict our analysis without loss of generality to $\Scal$ and, using that the functions 
$\phi$, $\psi$ and $f$ are locally Lipschitz continuous and $\Scal$ is bounded, we can see
that hypothesis (H$_1$) holds. Moreover, due to \eqref{eq:ShmoothProfiles} we have that
$f(t,\omega^k(t),v^i)=f_v(x)\omega^k(t) \in D(A)$ and 
\begin{equation}
\left\|Af(t,\omega^k(t),v^i)\right\|_X \leq \|A f_i\|_X |\omega^k(t)| \leq \bar{C^i}
\end{equation}
for some constant $\bar{C}^i$ and due to the choice of $U_{[0,t_f]}$ we have that
$f(s,\omega^k(s),v^i)$ is differentiable for a.\,e. $s \in [0,t_f]$ and
\begin{equation}
 \left\|\frac{d}{ds} f(s,\omega^k(s),v^i) \right\|_X = \left\|\frac{d}{ds}
 f_i \omega^k(s) \right\|_X \leq \|f_i\|_X \left|\frac{d}{ds} \omega^k(s)\right| \leq \bar{L}^i
\end{equation}
for some constant $\bar{L}^i$. Noting that \eqref{eq:heatlumped} and thus also the abstract system
is linear, we may apply Proposition~\ref{prop:OnH2Auniformlybounded} to see that hypothesis (H$_2$) holds 
for every $k=0,1,2,\ldots$. Also, the bound in {\normalfont (H$_3$)} holds for all $k$, observing that
\begin{equation}
  \sup_{t\in(0,t_f)} \|f(t,\omega^k(t),v^i)\|_X \leq 
   \|f_i\|_X \|\omega^k(t)\|_\infty \leq M^i,
\end{equation}
for some constant $M^i$. Hence, for any such choices $[\omega^k,\alpha^k,y^k]$, any sequence 
$\varepsilon^k \to 0$, $\Delta t^k \to 0$ and any $\varepsilon>0$, the relaxation method terminates after finitely many
steps $\kappa$ with an $\varepsilon$-feasible solution $[u^*,v^*,z^*]$ of problem \eqref{eq:heatlumped}
satisfying the estimate
\begin{equation}
|J(\omega^\kappa,\alpha^\kappa,y^\kappa) - J(u^*,v^*,z^*)| \leq \varepsilon, 
\end{equation}
by Proposition~\ref{prop:convergenceWeaker}. The desired switching structure $\sigma\: [0,t_f] \to \{1,\ldots,N\}$ 
is finally given by $\sigma(t)=v^*(t)$.

\begin{example}\label{ex:heatrelax2d} To demonstrate the applicability of the
approach, we implemented the relaxation method for a test problem of the form 
\eqref{eq:heatlumped}--\eqref{eq:heatcost} with a two-dimensional rectangular domain 
$\Omega$ and the following parameters.

Let $\Omega=[0,L_\xi] \times [0,L_\zeta]$, $\rho=0.01$, $L_\xi=1$, $L_\zeta=2$ and $t_f=15$ and 
suppose that there are given 9 actuator locations $x_i$ with the positions given by $(\xi_j,\zeta_k)
\in \Omega$, where 
\begin{equation} 
\xi_j = \frac{j+0.005 L_\xi}{4},~\zeta_k = \frac{k+0.005 L_\zeta}{4},\quad
j,k=1,2,3.
\end{equation}
Further suppose that there is a point actuator for each of these locations $x_i$
which we model here by setting $f_i=B_i$ with
\begin{equation}\label{eq:pact}
B_i(x) = \frac{1}{\sqrt{2 \pi \epsilon}}e^{\frac{-(x_i-x)^2}{2\epsilon}}
\end{equation} 
for some small, but fixed $\epsilon>0$. Note that $\int_{\Omega} B_i(x)\,dx = 1$
and that $B_i(x)$ converges to the Dirac delta function $\delta(x-x_i)$ as
$\epsilon \to 0$. 

As initial data we take 
\begin{equation}
z_0(\xi,\zeta)=10\sin(\pi \xi)10\sin(\pi \zeta) 
\end{equation}
and as parameters in the cost function we take $\lambda_1=2$ and
$\lambda_2=\frac{1}{500}$.

We have chosen these numerical values to match as closely as possible the 
two-dimensional example in \cite{IftimeDemetriou2009} motivated by thermal manufacturing. 
The only difference is that the pointwise actuators $\delta(x-x_i)$ were approximated
in \cite{IftimeDemetriou2009} by indicator functions of an epsilon environment while we choose 
here a smoother approximation in view of \eqref{eq:ShmoothProfiles}. Regarding a direct 
treatment of $\delta(x-x_i)$ as an unbounded control operator instead of using
the bounded approximation \eqref{eq:pact}, see the comments in Section~\ref{sec:final}.

The solution of the relaxed optimal control problem \eqref{eq:relaxedconvexifiedproblem}
has been computed numerically. We discretized the state equation \eqref{eq:heatlumped} in 
space using a standard Galerkin approach with triangular elements and linear Ansatz-functions.
We eliminated one control by setting $\tilde{\alpha}_i(t)=\alpha_i(t)$, $i=1,\ldots,N-1$,
where we then get $\alpha_N(t)=1-\sum_{i=1}^{N-1}\tilde{\alpha}_i(t)$ using the constraint
$\sum_{i=1}^N \alpha_i(t) = 1$, $t \in [0,t_f]$.
This constraint is then always fulfilled and the condition
 $\alpha_N(t)\in[0,1]$, $t \geq t_0$,
is equivalent to imposing that $\tilde{\alpha}_i \in [0,1]$ and 
$\sum_{i=1}^{N-1} \tilde{\alpha}_i - 1 \leq 0$.
The resulting semi-discretized control problem was solved with Bock's direct multiple shooting 
method \cite{Bock1984,Leineweber2003b} implemented in the software-package MUSCOD-II. The control functions $\omega$ and 
$\tilde{\alpha}$ were chosen as piecewise constant and initialized with $\omega(t)=0$ 
and $\tilde{\alpha}_i(t)=\frac19$, $t\in [0,t_f]$, $i=1,\ldots,8$. 

The computations were made for an unstructured grid with 162 triangular elements and 8, 16 and 32 equidistant shooting 
intervals. Time integration was carried out by a BDF-method and sensitivities were computed using internal numerical differentiation.
Error estimates for these methods provide an accuracy of some $\varepsilon_1 \geq 0$ for the so obtained approximations of
the mild solutions.

We implemented Algorithm~\ref{algo:iter} where we compute in step {\footnotesize $3$:} solutions satisfying first order 
necessary conditions with an accuracy of $\varepsilon_2$. The above discussion of the abstract example thus applies. In particular, using
that $\lambda_2>0$, we obtain the existence of a bounded set $\Scal$ containing the iterates $[\omega^k,\alpha^k,y^k]$. 
We adaptively solved the relaxed problem on a common control discretization grid $\Gcal^k$ for $u$ and $\tilde{\alpha}$.
For the computations, we have chosen $\varepsilon_1=\text{1.0E-04}$ and used bisection 
for refinements of the control grids $\Gcal^k$ in step {\footnotesize $7$:}. Thus, $\varepsilon^k$ is given implicitly as a function of
$\varepsilon_1$, $\varepsilon_2$ and $\Delta t^k$. This construction ensures that $\varepsilon^k \to 0$ as $k \to \infty$.

\begin{table}[t]
\renewcommand*\arraystretch{1.7}
\footnotesize
 \begin{longtable*}{|c|c|c|c|c|c|} 
\hline
 \footnotesize $k$ 
&
\footnotesize
$\Delta t^k_{\max}$
&
\footnotesize $J^k_\rel=J(\omega^k,\tilde{\alpha}^k,y^k)$
&
\footnotesize
$J^k=J(u^k,v^k,z^k)$
&
\footnotesize 
Error $({J^{2}_{\rel}})^{-1} |J^{2}_{\rel}-J^k|$ 
\\ 
\hline 

  0 & 1.8750 & 5.634024E+04 & 1.283813E+05  & 2.9809  \\
  1 & 0.9375 & 4.190360E+04 & 7.080185E+04  & 1.1955  \\
  2 & 0.4688 & 3.224914E+04 & 6.175488E+04  & 0.9149  \\
\hline
 \end{longtable*}
\normalsize \mbox{}
\caption{Performance of the relaxation method for
Example~\ref{ex:heatrelax2d}.}
\label{table:heatrelax2d}
\end{table}

The performance of the relaxation method is summarized in Table~\ref{table:heatrelax2d}. We see
that the relative error of the mixed-integer solution compared with the best found relaxed solution
decreases with the grid refinements in accordance with Proposition~\ref{prop:convergenceWeaker}. 
The best found controls and the evolution of the state norm of the corresponding solution
are displayed in Figure~\ref{fig:heatrelax2d}. We see the rounding error in form of an 
overshooting behavior when comparing the evolution of the $L^2(\Omega)$-norm of the relaxed
and the mixed-integer solution. This effect decreases with the size of the time discretization step 
size. The cost corresponding to the best found solution is 6175. Unfortunately, 
\cite{IftimeDemetriou2009} does not report the cost of the best found solution, but a cumulative 
$L^2(0,15;L^2(\Omega))$-norm of $90.27$. The cumulative $L^2(0,15;L^2(\Omega))$-norm of our
best found solution is $78.58$.

\begin{figure}[t]
\vspace*{-1em}
\scalebox{0.73}{
\renewcommand{\gplbacktext}{}
\renewcommand{\gplfronttext}{}
\begingroup
  \ifGPblacktext
    \def\colorrgb#1{}%
    \def\colorgray#1{}%
  \else
    \ifGPcolor
      \def\colorrgb#1{\color[rgb]{#1}}%
      \def\colorgray#1{\color[gray]{#1}}%
      \expandafter\def\csname LTw\endcsname{\color{white}}%
      \expandafter\def\csname LTb\endcsname{\color{black}}%
      \expandafter\def\csname LTa\endcsname{\color{black}}%
      \expandafter\def\csname LT0\endcsname{\color[rgb]{1,0,0}}%
      \expandafter\def\csname LT1\endcsname{\color[rgb]{0,1,0}}%
      \expandafter\def\csname LT2\endcsname{\color[rgb]{0,0,1}}%
      \expandafter\def\csname LT3\endcsname{\color[rgb]{1,0,1}}%
      \expandafter\def\csname LT4\endcsname{\color[rgb]{0,1,1}}%
      \expandafter\def\csname LT5\endcsname{\color[rgb]{1,1,0}}%
      \expandafter\def\csname LT6\endcsname{\color[rgb]{0,0,0}}%
      \expandafter\def\csname LT7\endcsname{\color[rgb]{1,0.3,0}}%
      \expandafter\def\csname LT8\endcsname{\color[rgb]{0.5,0.5,0.5}}%
    \else
      \def\colorrgb#1{\color{black}}%
      \def\colorgray#1{\color[gray]{#1}}%
      \expandafter\def\csname LTw\endcsname{\color{white}}%
      \expandafter\def\csname LTb\endcsname{\color{black}}%
      \expandafter\def\csname LTa\endcsname{\color{black}}%
      \expandafter\def\csname LT0\endcsname{\color{black}}%
      \expandafter\def\csname LT1\endcsname{\color{black}}%
      \expandafter\def\csname LT2\endcsname{\color{black}}%
      \expandafter\def\csname LT3\endcsname{\color{black}}%
      \expandafter\def\csname LT4\endcsname{\color{black}}%
      \expandafter\def\csname LT5\endcsname{\color{black}}%
      \expandafter\def\csname LT6\endcsname{\color{black}}%
      \expandafter\def\csname LT7\endcsname{\color{black}}%
      \expandafter\def\csname LT8\endcsname{\color{black}}%
    \fi
  \fi
  \setlength{\unitlength}{0.0500bp}%
  \begin{picture}(4384.00,4533.33)%
    \gplgaddtomacro\gplbacktext{%
      \csname LTb\endcsname%
      \put(264,618){\makebox(0,0)[r]{\strut{} 0}}%
      \put(264,1194){\makebox(0,0)[r]{\strut{} 1}}%
      \put(396,340){\makebox(0,0){\strut{}0}}%
      \put(1154,340){\makebox(0,0){\strut{}3}}%
      \put(1912,340){\makebox(0,0){\strut{}6}}%
      \put(2670,340){\makebox(0,0){\strut{}9}}%
      \put(3428,340){\makebox(0,0){\strut{}12}}%
      \put(4186,340){\makebox(0,0){\strut{}15}}%
      \put(2291,120){\makebox(0,0){\strut{}Time $t$}}%
    }%
    \gplgaddtomacro\gplfronttext{%
    }%
    \gplgaddtomacro\gplbacktext{%
      \csname LTb\endcsname%
      \put(264,1468){\makebox(0,0)[r]{\strut{} 0}}%
      \put(264,2044){\makebox(0,0)[r]{\strut{} 1}}%
      \put(396,1190){\makebox(0,0){\strut{} }}%
      \put(1154,1190){\makebox(0,0){\strut{} }}%
      \put(1912,1190){\makebox(0,0){\strut{} }}%
      \put(2670,1190){\makebox(0,0){\strut{} }}%
      \put(3428,1190){\makebox(0,0){\strut{} }}%
      \put(4186,1190){\makebox(0,0){\strut{} }}%
    }%
    \gplgaddtomacro\gplfronttext{%
    }%
    \gplgaddtomacro\gplbacktext{%
      \csname LTb\endcsname%
      \put(264,2318){\makebox(0,0)[r]{\strut{} 0}}%
      \put(264,2894){\makebox(0,0)[r]{\strut{} 1}}%
      \put(396,2040){\makebox(0,0){\strut{} }}%
      \put(1154,2040){\makebox(0,0){\strut{} }}%
      \put(1912,2040){\makebox(0,0){\strut{} }}%
      \put(2670,2040){\makebox(0,0){\strut{} }}%
      \put(3428,2040){\makebox(0,0){\strut{} }}%
      \put(4186,2040){\makebox(0,0){\strut{} }}%
    }%
    \gplgaddtomacro\gplfronttext{%
    }%
    \gplgaddtomacro\gplbacktext{%
      \csname LTb\endcsname%
      \put(264,3168){\makebox(0,0)[r]{\strut{} 0}}%
      \put(264,3744){\makebox(0,0)[r]{\strut{} 1}}%
      \put(396,2890){\makebox(0,0){\strut{} }}%
      \put(1154,2890){\makebox(0,0){\strut{} }}%
      \put(1912,2890){\makebox(0,0){\strut{} }}%
      \put(2670,2890){\makebox(0,0){\strut{} }}%
      \put(3428,2890){\makebox(0,0){\strut{} }}%
      \put(4186,2890){\makebox(0,0){\strut{} }}%
      \put(2291,4022){\makebox(0,0){\strut{}$\tilde{\alpha}_i^*(\cdot)$, $v_i^*(\cdot)$, $i=1,\ldots,4$}}%
    }%
    \gplgaddtomacro\gplfronttext{%
    }%
    \gplbacktext{}
    \put(0,0){\includegraphics{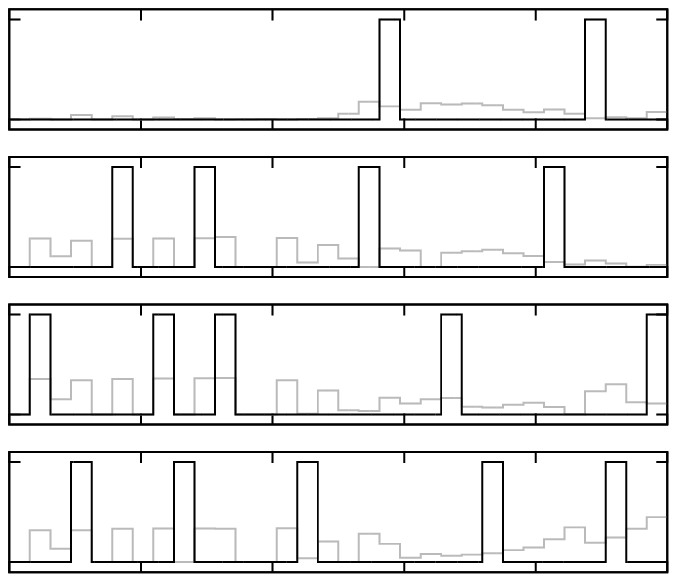}}%
    \gplfronttext{}
  \end{picture}%
\endgroup}~\scalebox{0.73}{
\renewcommand{\gplbacktext}{}
\renewcommand{\gplfronttext}{}
\begingroup
  \ifGPblacktext
    \def\colorrgb#1{}%
    \def\colorgray#1{}%
  \else
    \ifGPcolor
      \def\colorrgb#1{\color[rgb]{#1}}%
      \def\colorgray#1{\color[gray]{#1}}%
      \expandafter\def\csname LTw\endcsname{\color{white}}%
      \expandafter\def\csname LTb\endcsname{\color{black}}%
      \expandafter\def\csname LTa\endcsname{\color{black}}%
      \expandafter\def\csname LT0\endcsname{\color[rgb]{1,0,0}}%
      \expandafter\def\csname LT1\endcsname{\color[rgb]{0,1,0}}%
      \expandafter\def\csname LT2\endcsname{\color[rgb]{0,0,1}}%
      \expandafter\def\csname LT3\endcsname{\color[rgb]{1,0,1}}%
      \expandafter\def\csname LT4\endcsname{\color[rgb]{0,1,1}}%
      \expandafter\def\csname LT5\endcsname{\color[rgb]{1,1,0}}%
      \expandafter\def\csname LT6\endcsname{\color[rgb]{0,0,0}}%
      \expandafter\def\csname LT7\endcsname{\color[rgb]{1,0.3,0}}%
      \expandafter\def\csname LT8\endcsname{\color[rgb]{0.5,0.5,0.5}}%
    \else
      \def\colorrgb#1{\color{black}}%
      \def\colorgray#1{\color[gray]{#1}}%
      \expandafter\def\csname LTw\endcsname{\color{white}}%
      \expandafter\def\csname LTb\endcsname{\color{black}}%
      \expandafter\def\csname LTa\endcsname{\color{black}}%
      \expandafter\def\csname LT0\endcsname{\color{black}}%
      \expandafter\def\csname LT1\endcsname{\color{black}}%
      \expandafter\def\csname LT2\endcsname{\color{black}}%
      \expandafter\def\csname LT3\endcsname{\color{black}}%
      \expandafter\def\csname LT4\endcsname{\color{black}}%
      \expandafter\def\csname LT5\endcsname{\color{black}}%
      \expandafter\def\csname LT6\endcsname{\color{black}}%
      \expandafter\def\csname LT7\endcsname{\color{black}}%
      \expandafter\def\csname LT8\endcsname{\color{black}}%
    \fi
  \fi
  \setlength{\unitlength}{0.0500bp}%
  \begin{picture}(4384.00,4533.33)%
    \gplgaddtomacro\gplbacktext{%
      \csname LTb\endcsname%
      \put(264,618){\makebox(0,0)[r]{\strut{} 0}}%
      \put(264,1194){\makebox(0,0)[r]{\strut{} 1}}%
      \put(396,340){\makebox(0,0){\strut{}0}}%
      \put(1154,340){\makebox(0,0){\strut{}3}}%
      \put(1912,340){\makebox(0,0){\strut{}6}}%
      \put(2670,340){\makebox(0,0){\strut{}9}}%
      \put(3428,340){\makebox(0,0){\strut{}12}}%
      \put(4186,340){\makebox(0,0){\strut{}15}}%
      \put(2291,120){\makebox(0,0){\strut{}Time $t$}}%
    }%
    \gplgaddtomacro\gplfronttext{%
    }%
    \gplgaddtomacro\gplbacktext{%
      \csname LTb\endcsname%
      \put(264,1468){\makebox(0,0)[r]{\strut{} 0}}%
      \put(264,2044){\makebox(0,0)[r]{\strut{} 1}}%
      \put(396,1190){\makebox(0,0){\strut{} }}%
      \put(1154,1190){\makebox(0,0){\strut{} }}%
      \put(1912,1190){\makebox(0,0){\strut{} }}%
      \put(2670,1190){\makebox(0,0){\strut{} }}%
      \put(3428,1190){\makebox(0,0){\strut{} }}%
      \put(4186,1190){\makebox(0,0){\strut{} }}%
    }%
    \gplgaddtomacro\gplfronttext{%
    }%
    \gplgaddtomacro\gplbacktext{%
      \csname LTb\endcsname%
      \put(264,2318){\makebox(0,0)[r]{\strut{} 0}}%
      \put(264,2894){\makebox(0,0)[r]{\strut{} 1}}%
      \put(396,2040){\makebox(0,0){\strut{} }}%
      \put(1154,2040){\makebox(0,0){\strut{} }}%
      \put(1912,2040){\makebox(0,0){\strut{} }}%
      \put(2670,2040){\makebox(0,0){\strut{} }}%
      \put(3428,2040){\makebox(0,0){\strut{} }}%
      \put(4186,2040){\makebox(0,0){\strut{} }}%
    }%
    \gplgaddtomacro\gplfronttext{%
    }%
    \gplgaddtomacro\gplbacktext{%
      \csname LTb\endcsname%
      \put(264,3168){\makebox(0,0)[r]{\strut{} 0}}%
      \put(264,3744){\makebox(0,0)[r]{\strut{} 1}}%
      \put(396,2890){\makebox(0,0){\strut{} }}%
      \put(1154,2890){\makebox(0,0){\strut{} }}%
      \put(1912,2890){\makebox(0,0){\strut{} }}%
      \put(2670,2890){\makebox(0,0){\strut{} }}%
      \put(3428,2890){\makebox(0,0){\strut{} }}%
      \put(4186,2890){\makebox(0,0){\strut{} }}%
      \put(2291,4019){\makebox(0,0){\strut{}$\tilde{\alpha}_i^*(\cdot)$, $v_i^*(\cdot)$, $i=5,\ldots,8$}}%
    }%
    \gplgaddtomacro\gplfronttext{%
    }%
    \gplbacktext
    \put(0,0){\includegraphics{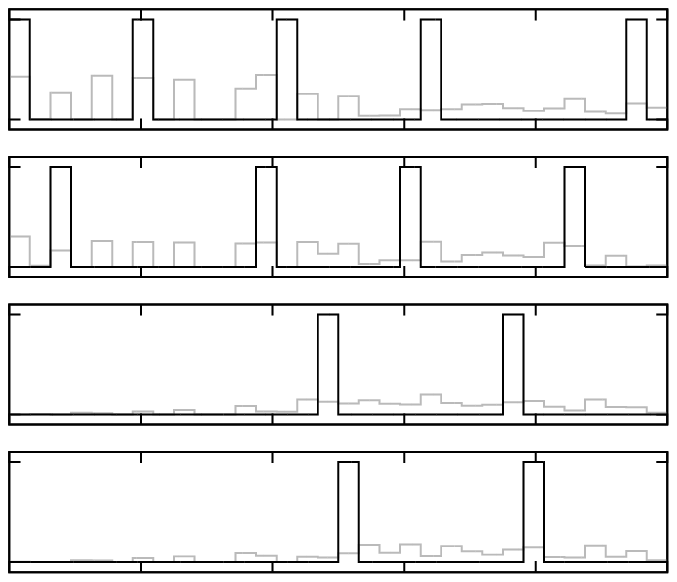}}%
    \gplfronttext
  \end{picture}%
\endgroup
}\\
\scalebox{0.73}{
\renewcommand{\gplbacktext}{}
\renewcommand{\gplfronttext}{}
\begingroup
  \ifGPblacktext
    \def\colorrgb#1{}%
    \def\colorgray#1{}%
  \else
    \ifGPcolor
      \def\colorrgb#1{\color[rgb]{#1}}%
      \def\colorgray#1{\color[gray]{#1}}%
      \expandafter\def\csname LTw\endcsname{\color{white}}%
      \expandafter\def\csname LTb\endcsname{\color{black}}%
      \expandafter\def\csname LTa\endcsname{\color{black}}%
      \expandafter\def\csname LT0\endcsname{\color[rgb]{1,0,0}}%
      \expandafter\def\csname LT1\endcsname{\color[rgb]{0,1,0}}%
      \expandafter\def\csname LT2\endcsname{\color[rgb]{0,0,1}}%
      \expandafter\def\csname LT3\endcsname{\color[rgb]{1,0,1}}%
      \expandafter\def\csname LT4\endcsname{\color[rgb]{0,1,1}}%
      \expandafter\def\csname LT5\endcsname{\color[rgb]{1,1,0}}%
      \expandafter\def\csname LT6\endcsname{\color[rgb]{0,0,0}}%
      \expandafter\def\csname LT7\endcsname{\color[rgb]{1,0.3,0}}%
      \expandafter\def\csname LT8\endcsname{\color[rgb]{0.5,0.5,0.5}}%
    \else
      \def\colorrgb#1{\color{black}}%
      \def\colorgray#1{\color[gray]{#1}}%
      \expandafter\def\csname LTw\endcsname{\color{white}}%
      \expandafter\def\csname LTb\endcsname{\color{black}}%
      \expandafter\def\csname LTa\endcsname{\color{black}}%
      \expandafter\def\csname LT0\endcsname{\color{black}}%
      \expandafter\def\csname LT1\endcsname{\color{black}}%
      \expandafter\def\csname LT2\endcsname{\color{black}}%
      \expandafter\def\csname LT3\endcsname{\color{black}}%
      \expandafter\def\csname LT4\endcsname{\color{black}}%
      \expandafter\def\csname LT5\endcsname{\color{black}}%
      \expandafter\def\csname LT6\endcsname{\color{black}}%
      \expandafter\def\csname LT7\endcsname{\color{black}}%
      \expandafter\def\csname LT8\endcsname{\color{black}}%
    \fi
  \fi
  \setlength{\unitlength}{0.0500bp}%
  \begin{picture}(4384.00,4533.33)%
    \gplgaddtomacro\gplbacktext{%
      \csname LTb\endcsname%
      \put(264,904){\makebox(0,0)[r]{\strut{}-40}}%
      \put(264,1613){\makebox(0,0)[r]{\strut{}-20}}%
      \put(264,2321){\makebox(0,0)[r]{\strut{} 0}}%
      \put(264,3029){\makebox(0,0)[r]{\strut{} 20}}%
      \put(264,3738){\makebox(0,0)[r]{\strut{} 40}}%
      \put(396,330){\makebox(0,0){\strut{} 0}}%
      \put(1154,330){\makebox(0,0){\strut{} 3}}%
      \put(1912,330){\makebox(0,0){\strut{} 6}}%
      \put(2670,330){\makebox(0,0){\strut{} 9}}%
      \put(3428,330){\makebox(0,0){\strut{} 12}}%
      \put(4186,330){\makebox(0,0){\strut{} 15}}%
      \put(2291,110){\makebox(0,0){\strut{}Time t}}%
      \put(2291,4312){\makebox(0,0){\strut{}Ordinary control $u^*(\cdot)$}}%
    }%
    \gplgaddtomacro\gplfronttext{%
    }%
    \gplbacktext{}
    \put(0,0){\includegraphics{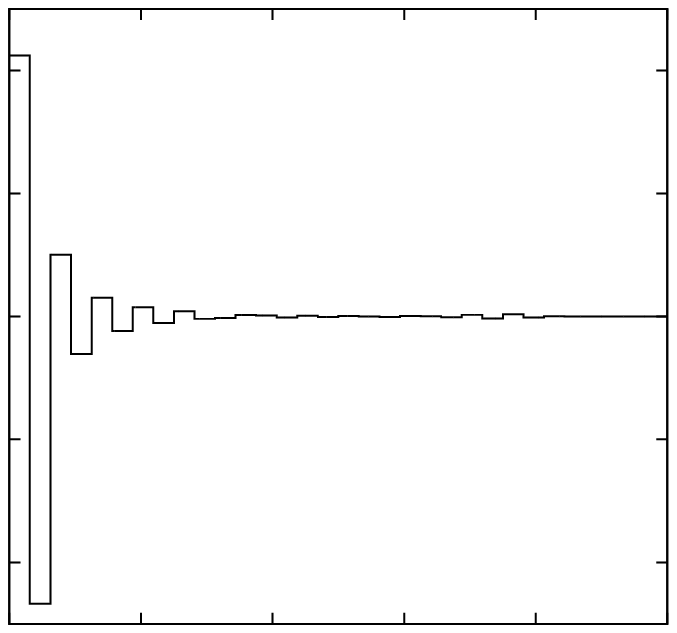}}%
    \gplfronttext{}
  \end{picture}%
\endgroup
}~\scalebox{0.73}{
\renewcommand{\gplbacktext}{}
\renewcommand{\gplfronttext}{}
\begingroup
  \ifGPblacktext
    \def\colorrgb#1{}%
    \def\colorgray#1{}%
  \else
    \ifGPcolor
      \def\colorrgb#1{\color[rgb]{#1}}%
      \def\colorgray#1{\color[gray]{#1}}%
      \expandafter\def\csname LTw\endcsname{\color{white}}%
      \expandafter\def\csname LTb\endcsname{\color{black}}%
      \expandafter\def\csname LTa\endcsname{\color{black}}%
      \expandafter\def\csname LT0\endcsname{\color[rgb]{1,0,0}}%
      \expandafter\def\csname LT1\endcsname{\color[rgb]{0,1,0}}%
      \expandafter\def\csname LT2\endcsname{\color[rgb]{0,0,1}}%
      \expandafter\def\csname LT3\endcsname{\color[rgb]{1,0,1}}%
      \expandafter\def\csname LT4\endcsname{\color[rgb]{0,1,1}}%
      \expandafter\def\csname LT5\endcsname{\color[rgb]{1,1,0}}%
      \expandafter\def\csname LT6\endcsname{\color[rgb]{0,0,0}}%
      \expandafter\def\csname LT7\endcsname{\color[rgb]{1,0.3,0}}%
      \expandafter\def\csname LT8\endcsname{\color[rgb]{0.5,0.5,0.5}}%
    \else
      \def\colorrgb#1{\color{black}}%
      \def\colorgray#1{\color[gray]{#1}}%
      \expandafter\def\csname LTw\endcsname{\color{white}}%
      \expandafter\def\csname LTb\endcsname{\color{black}}%
      \expandafter\def\csname LTa\endcsname{\color{black}}%
      \expandafter\def\csname LT0\endcsname{\color{black}}%
      \expandafter\def\csname LT1\endcsname{\color{black}}%
      \expandafter\def\csname LT2\endcsname{\color{black}}%
      \expandafter\def\csname LT3\endcsname{\color{black}}%
      \expandafter\def\csname LT4\endcsname{\color{black}}%
      \expandafter\def\csname LT5\endcsname{\color{black}}%
      \expandafter\def\csname LT6\endcsname{\color{black}}%
      \expandafter\def\csname LT7\endcsname{\color{black}}%
      \expandafter\def\csname LT8\endcsname{\color{black}}%
    \fi
  \fi
  \setlength{\unitlength}{0.0500bp}%
  \begin{picture}(4384.00,4533.33)%
    \gplgaddtomacro\gplbacktext{%
      \csname LTb\endcsname%
      \put(264,550){\makebox(0,0)[r]{\strut{} 0}}%
      \put(264,993){\makebox(0,0)[r]{\strut{} 10}}%
      \put(264,1436){\makebox(0,0)[r]{\strut{} 20}}%
      \put(264,1878){\makebox(0,0)[r]{\strut{} 30}}%
      \put(264,2321){\makebox(0,0)[r]{\strut{} 40}}%
      \put(264,2764){\makebox(0,0)[r]{\strut{} 50}}%
      \put(264,3207){\makebox(0,0)[r]{\strut{} 60}}%
      \put(264,3649){\makebox(0,0)[r]{\strut{} 70}}%
      \put(264,4092){\makebox(0,0)[r]{\strut{} 80}}%
      \put(396,330){\makebox(0,0){\strut{} 0}}%
      \put(1154,330){\makebox(0,0){\strut{} 3}}%
      \put(1912,330){\makebox(0,0){\strut{} 6}}%
      \put(2670,330){\makebox(0,0){\strut{} 9}}%
      \put(3428,330){\makebox(0,0){\strut{} 12}}%
      \put(4186,330){\makebox(0,0){\strut{} 15}}%
      \put(2291,110){\makebox(0,0){\strut{}Time t}}%
      \put(2291,4312){\makebox(0,0){\strut{}Statenorm evolution $\|y(\cdot)\|_X$}}%
    }%
    \gplgaddtomacro\gplfronttext{%
      \csname LTb\endcsname%
      \put(3199,3919){\makebox(0,0)[r]{\strut{}uncontrolled}}%
      \csname LTb\endcsname%
      \put(3199,3699){\makebox(0,0)[r]{\strut{}relaxed}}%
      \csname LTb\endcsname%
      \put(3199,3479){\makebox(0,0)[r]{\strut{}mixed-integer}}%
    }%
    \gplbacktext
    \put(0,0){\includegraphics{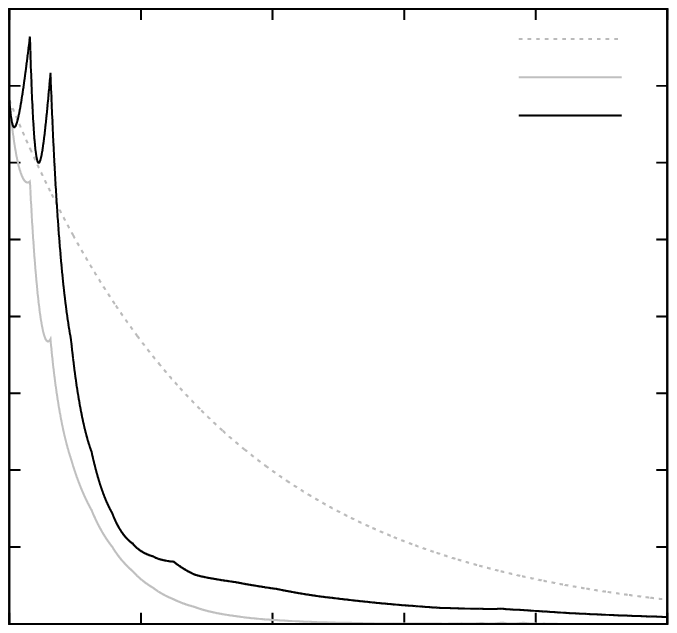}}%
    \gplfronttext
  \end{picture}%
\endgroup
}
\caption{Numerical results for Example~\ref{ex:heatrelax2d}. The upper figures show
the best found integer controls $v_i^*(\cdot)$ and their relaxation $\tilde{\alpha}_i^*(\cdot)$,
from bottom to top, $i=1,\ldots,4$ (left) and $i=5,\ldots,8$ (right). Control $v_9^*(\cdot)$ is 
defined by $v_9^*(t)=1-\sum_{i=1}^8 v_i^*(t)$, $t \in [0,15]$. 
The lower figures show the corresponding ordinary control $u^*(\cdot)$ (left) and the 
evolution of the state norm (right).}
\label{fig:heatrelax2d} 
\end{figure}
\end{example}

\subsection{A semilinear reaction-diffusion system} Let $\Omega$ be a bounded domain in
$\RR^n$ with a smooth boundary $\Gamma$ and consider the classical Lotka-Volterra system with diffusion
\begin{equation}\label{eq:lotkavolterra}
\left\{\begin{aligned}
&\frac{\partial z_1}{\partial t}(x,t)\,{-}\,d_1\sum_{j=1}^n \frac{\partial^2
z_1}{\partial x^2_j}(x,t) = z_1(x,t)(a_1
\,{-}\,b_1v(t)\,{-}\,c_1z_2(x,t))~ \text{in}~Q\\
&\frac{\partial z_2}{\partial t}(x,t)\,{-}\,d_2\sum_{j=1}^n \frac{\partial^2 z_2}{\partial x^2_j}(x,t) =
z_2(x,t)(a_2\,{-}\,b_2v(t)\,{-}\,c_2z_1(x,t))~ \text{in}~Q\\ &
\frac{\partial z_1}{\partial \nu}(x,t)=\frac{\partial z_2}{\partial \nu}(x,t)=0~ \text{on}~\Sigma\\ 
&\vphantom{\frac{\partial  z_1}{\partial \nu}} z_1(x,0)=z_{1,0}(x),~z_2(x,0)=z_{2,0}(x)~ \text{in}~\Omega
\end{aligned}\right.
\end{equation}
with constants $a_i,b_i,c_i,d_i>0$, $i=1,2$, domains $Q=\Omega \times [0,t_f]$, $\Sigma=\Gamma \times [0,t_f]$ 
and control $0 \leq v(t) \leq 1$. System \eqref{eq:lotkavolterra} describes the interaction of two
populations $z_1$ and $z_2$, both spatially distributed and diffusing in
$\Omega$. The initial distribution $z_{1,0},z_{2,0}$ at $t=0$ is assumed to be 
non-negative. The boundary conditions then imply that the populations $z_1$
and $z_2$ are confined in $\Omega$ for all $t\geq 0$. The function $v$ models 
a control of the system and we shall investigate to approximate optimal controls
$v^*(t)$ taking values in $\{0,1\}$ as to minimize the distance of the 
population $(z_1,z_2)$ to its uncontrolled ($v=0$) steady state distribution
$(\bar{z}_1,\bar{z}_2)$ given by the constant functions 
\begin{equation*}
 \bar{z}_1(x)=\frac{a_2}{c_2},~\bar{z}_2(x)=\frac{a_1}{c_1},\quad x\in\Omega. 
\end{equation*}

In order to bring the system into abstract form \eqref{eq:sysuc}, set
$X=L^2(\Omega)\times L^2(\Omega)$, $U=\Uad=\{\}$, $V=\RR$, $\Vad=\{0,1\}$,
define the operator $A\: D(A) \to X$ by
\begin{equation*}
\begin{aligned}
&D(A)=\{(z_1,z_2) \in H^2(\Omega) \times H^2(\Omega) : \frac{\partial
z_1}{\partial \nu}(x,t)=\frac{\partial z_2}{\partial \nu}(x,t)=0,\quad
\text{on}~\Gamma\}\\
&A(z_1,z_2)(x)=(d_1 \sum_{j=1}^n \frac{\partial^2
z_1}{\partial x^2_j}(x),d_2 \sum_{j=1}^n \frac{\partial^2 z_2}{\partial
x^2_j}(x)),~(z_1,z_2) \in D(A),
\end{aligned}
\end{equation*}
define the non-linear function $f\: X \times U \times V = X \times V
\to X$ by
\begin{equation*}
f((z_1,z_2),v)(x)=(z_1(x)(a_1-b_1v-c_1z_2(x)),z_2(x)(a_2-b_2v-c_2z_2(x))
\end{equation*}
and define the cost functions $\phi$ and $\psi$ by
\begin{equation}\label{eq:LotkaVolterraCost}
\phi((z_1,z_2))=0,\quad \psi((z_1,z_2)) = \int_\Omega \|z_1(x)-\bar{z}_1(x)\|^2 +
\|z_2(x)-\bar{z}_2(x)\|^2\,dx.
\end{equation}
We choose $X_{[0,t_f]}=C([0,t_f];X)$ and $V_{[0,t_f]}=PC(0,t_f;\RR)$.

It is well-known, that $(A,D(A))$ is the generator of an analytic
semigroup on $X$ and that for any non-negative initial data
$z_{1,0},z_{1,0} \in X$, the system \eqref{eq:lotkavolterra} has a non-negative
unique mild solution $(z_1,z_2) \in C([0,t_f],X \times X)$ for every $v \in
L^\infty(0,t_f;[0,1])$. Moreover, for data $z_{1,0},z_{1,0} \in D(A)$
and $v \in C^{0,\vartheta}_{\pw}(0,t_f;[0,1])$, $\vartheta>0$, this solution is classical and satisfies 
\begin{equation}\label{eq:regy}
 (z_1,z_2) \in C([0,t_f];D(A) \times D(A)) \cap H^1([0,t_f]; X \times X).
\end{equation}
Existence (local in time) and uniqueness follows from classical theory for semilinear parabolic 
equations, see, e.\,g., \cite[Chapter 6]{Pazy1983}. Global existence results for
\eqref{eq:lotkavolterra} are obtained by a-priori bounds on the solution using
contracting rectangles \cite{Brown1980}.

Assume that the initial data satisfies $z_{1,0},z_{1,0} \in D(A)$ and that $[\alpha^k,y^k]$
is a sequence of feasible solutions to the corresponding relaxed problem 
\eqref{eq:relaxedconvexifiedproblem} in a bounded set $\Scal \subset \tilde{V}_{[0,t_f]} \times X_{[0,t_f]}$. 
We want to discuss again the assumptions of Proposition~\ref{prop:convergenceWeaker}.

Hypothesis (H$_1$) holds by the same arguments as in the previous example. Moreover, we claim that 
hypothesis {\normalfont (H$_2$)} holds. Let $\bar{M}$ be the growth bound of
$\{T(t)\}_{t \geq 0}$ on $[0,t_f]$.
By analyticity of $\{T(t)\}_{t \geq 0}$, we have for every feasible 
$y \in X_{[0,t_f]}$ that $s$ almost everywhere in $(0,t_f)$,
\begin{equation}\label{eq:nonlinearest}
\frac{d}{ds}T(t-s)f(y(s),v^i) = -A T(t-s)f(y(s),v^i) +
T(t-s)f_y(y(s),v^i)y_s(s),
\end{equation}
where $f_y=\frac{d}{dy}f$ and $y_s=\frac{d}{ds}y$. Using that $y^k(s) \in D(A)
\times D(A)$ for all $s \in [0,t_f]$ and $f\: D(A) \to D(A)$, 
we see that
\begin{equation}
\|-A T(t-s)f(y^k(s),v^i)\|_X \leq \|T(t-s)\|_{\Lcal(X)}\|Af(y^k(s),v^i)\|_X \leq C_1^{k,i}
\end{equation}
for some constants $C_1^{k,i}$. Using that $f$ is a smooth function, we see that
\begin{equation}
\|T(t-s)f_y(y^k(s),v^i)y^k_s(s)\|_X \leq \bar{M}\|f_y(y^k(s),v^i)\|_X\|y^k_s(s)\|_X \leq C_2^{k,i}.
\end{equation}
Thus \eqref{eq:nonlinearest} yields the estimate
\begin{equation}\label{eq:H2estex2}
\left\|\frac{d}{ds}T(t-s)f(y(s),v^i)\right\|_X \leq C_1^{k,i} + C_2^{k,i}
\end{equation}
for $s \in [0,t_f]$ a.\,e. By well-posedness of the problem \eqref{eq:lotkavolterra} for every $v \in V_{[0,t_f]}$
and using the boundedness of $\Scal$, we get an estimate
\begin{equation}\label{eq:fboundexp2}
\sup_{t \in [0,t_f]} \|f(t,y^k(t),v^i)\|_X \leq M^{i},
\end{equation} 
for some constants $M^i$ verifying hypothesis {\normalfont (H$_3$)}. Further, using (H$_1$) and the boundedness of $\Scal$,
$\|f_y(y^k(s),v^i)\|_X$ and $\|Af(y^k(s),v^i)\|_X$ can be bounded independently of $k$. 
Using the state equation \eqref{eq:relaxedconvexifiedproblemB}, we get
\begin{equation}
\|y^k_s(s)\|_X \leq \|A y^k(s)\|_X + \|f(y^k(s),v^i)\|_X \leq \|y^k(s)\|_{D(A)} + M^{i}
\end{equation}
for a.\,e. $s \in [0,t_f]$. So, boundedness of $S$ and the regularity of $y$ in \eqref{eq:regy} implies 
that $\|y^k_s(s)\|_X$ can be bounded independently of $k$ for a.\,e. $s \in [0,t_f]$. Hence, the constants
$C^{k,i}_1$ and $C^{k,i}_2$ in \eqref{eq:H2estex2} can be chosen independently of $k$ and we can conclude from 
Proposition~\ref{prop:convergenceWeaker} that the relaxation methods terminates after $\kappa$ steps
with an integer solution $[v^*,z^*]$ satisfying the estimate
\begin{equation}
|J(\alpha^\kappa,y^\kappa) - J(v^*,z^*)| \leq \varepsilon,  
\end{equation}
for any sequences $\varepsilon^k \to 0$, $\Delta t^k \to 0$ and every $\varepsilon>0$.

\begin{table}[t]
\renewcommand*\arraystretch{1.5}
\footnotesize
 \begin{longtable*}{|c|c|c|c|c|} 
\hline
k
&
$\Delta t_{\max}$
&
Rel. Cost \footnotesize $J^k_\rel=J(\alpha^*,y^*)$
&
\footnotesize
$J^k=J(v^*,z^*)$
&
\footnotesize 
Error $({J^2_\rel})^{-1} |J^2_\rel-J^k|$ 
\\
\hline
 0 & 2.0000 & 7.066392E+01 & 8.287875E+01  & 0.4065 \\
 1 & 1.0000 & 5.978818E+01 & 6.958250E+01  & 0.1809 \\
 2 & 0.5000 & 5.892414E+01 & 5.875641E+01  & 0.0028 \\
\hline
 \end{longtable*}
\normalsize \mbox{}
\caption{Performance of the relaxation method for Example~\ref{ex:lotkadiff2d}.}
\label{table:lotkadiff2d}
\end{table}

\begin{example}\label{ex:lotkadiff2d} We applied the relaxation method 
to a semilinear test problem of the form \eqref{eq:lotkavolterra} again
for a two dimensional domain $\Omega$ with the following parameters. Let $\Omega$ being a circle with radius $1$ centered at $(1,1)$ and choose $a_1=a_2=c_1=c_2=1$,
$b_1=\frac{7}{10}$, $b_2=\frac{1}{2}$, $d_1=0.05$, $d_2=0.01$, initial data $z_{1,0},~z_{2,0} \in D(A)$ approximated by
$\tilde{z}_{1,0}(x)=\frac{1}{2} d_{\frac{1}{2}}(x-1)$,
$\tilde{z}_{2,0}(x)=\frac{7}{10} d_{\frac{1}{2}}(x-1)$, 
where $d_\epsilon(x)$ given by
\begin{equation}
d_\epsilon(x)=\frac{1}{\sqrt{2 \pi \epsilon}}e^{\frac{-x^2}{2\epsilon}}
\end{equation}
models a population concentrated at the origin. For $v(t)=0$, $t \geq 0$, the
solution $z_1(t,x)$, $z_2(t,x)$ converges asymptotically to a spatially constant 
and temporarily non-constant, periodic solution.

The computations for the optimal control are made by the same numerical method as in the 
previous example, but using a grid with 258 finite elements. For the performance of the relaxation 
method see Table~\ref{table:lotkadiff2d}. The best found controls and the evolution of the state norm of the corresponding solutions are displayed 
in Figure~\ref{fig:lotkadiff2d}. Again we see a decrease of the integer-approximation error 
in accordance with Proposition~\ref{prop:convergenceWeaker}. The best found integer control yields
a cost of 58.76.

\begin{figure}[t]
\scalebox{0.73}{
\renewcommand{\gplbacktext}{}
\renewcommand{\gplfronttext}{}
\begingroup
  \ifGPblacktext
    \def\colorrgb#1{}%
    \def\colorgray#1{}%
  \else
    \ifGPcolor
      \def\colorrgb#1{\color[rgb]{#1}}%
      \def\colorgray#1{\color[gray]{#1}}%
      \expandafter\def\csname LTw\endcsname{\color{white}}%
      \expandafter\def\csname LTb\endcsname{\color{black}}%
      \expandafter\def\csname LTa\endcsname{\color{black}}%
      \expandafter\def\csname LT0\endcsname{\color[rgb]{1,0,0}}%
      \expandafter\def\csname LT1\endcsname{\color[rgb]{0,1,0}}%
      \expandafter\def\csname LT2\endcsname{\color[rgb]{0,0,1}}%
      \expandafter\def\csname LT3\endcsname{\color[rgb]{1,0,1}}%
      \expandafter\def\csname LT4\endcsname{\color[rgb]{0,1,1}}%
      \expandafter\def\csname LT5\endcsname{\color[rgb]{1,1,0}}%
      \expandafter\def\csname LT6\endcsname{\color[rgb]{0,0,0}}%
      \expandafter\def\csname LT7\endcsname{\color[rgb]{1,0.3,0}}%
      \expandafter\def\csname LT8\endcsname{\color[rgb]{0.5,0.5,0.5}}%
    \else
      \def\colorrgb#1{\color{black}}%
      \def\colorgray#1{\color[gray]{#1}}%
      \expandafter\def\csname LTw\endcsname{\color{white}}%
      \expandafter\def\csname LTb\endcsname{\color{black}}%
      \expandafter\def\csname LTa\endcsname{\color{black}}%
      \expandafter\def\csname LT0\endcsname{\color{black}}%
      \expandafter\def\csname LT1\endcsname{\color{black}}%
      \expandafter\def\csname LT2\endcsname{\color{black}}%
      \expandafter\def\csname LT3\endcsname{\color{black}}%
      \expandafter\def\csname LT4\endcsname{\color{black}}%
      \expandafter\def\csname LT5\endcsname{\color{black}}%
      \expandafter\def\csname LT6\endcsname{\color{black}}%
      \expandafter\def\csname LT7\endcsname{\color{black}}%
      \expandafter\def\csname LT8\endcsname{\color{black}}%
    \fi
  \fi
  \setlength{\unitlength}{0.0500bp}%
  \begin{picture}(4384.00,4533.33)%
    \gplgaddtomacro\gplbacktext{%
      \csname LTb\endcsname%
      \put(264,803){\makebox(0,0)[r]{\strut{} 0}}%
      \put(264,1309){\makebox(0,0)[r]{\strut{} 0.2}}%
      \put(264,1815){\makebox(0,0)[r]{\strut{} 0.4}}%
      \put(264,2321){\makebox(0,0)[r]{\strut{} 0.6}}%
      \put(264,2827){\makebox(0,0)[r]{\strut{} 0.8}}%
      \put(264,3333){\makebox(0,0)[r]{\strut{} 1}}%
      \put(264,3839){\makebox(0,0)[r]{\strut{} 1.2}}%
      \put(396,330){\makebox(0,0){\strut{} 0}}%
      \put(1154,330){\makebox(0,0){\strut{} 3}}%
      \put(1911,330){\makebox(0,0){\strut{} 6}}%
      \put(2669,330){\makebox(0,0){\strut{} 9}}%
      \put(3426,330){\makebox(0,0){\strut{} 12}}%
      \put(4184,330){\makebox(0,0){\strut{} 15}}%
      \put(2290,110){\makebox(0,0){\strut{}Time t}}%
      \put(2290,4422){\makebox(0,0){\strut{}Relaxed and integer control}}%
    }%
    \gplgaddtomacro\gplfronttext{%
      \csname LTb\endcsname%
      \put(3197,3919){\makebox(0,0)[r]{\strut{}relaxed}}%
      \csname LTb\endcsname%
      \put(3197,3699){\makebox(0,0)[r]{\strut{}integer}}%
    }%
    \gplbacktext{}
    \put(0,0){\includegraphics{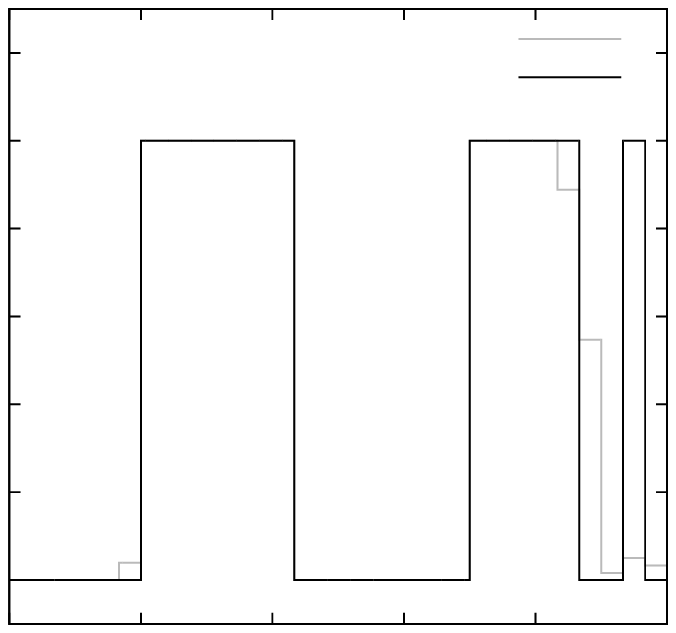}}%
    \gplfronttext{}
  \end{picture}%
\endgroup
}\scalebox{0.73}{
\renewcommand{\gplbacktext}{}
\renewcommand{\gplfronttext}{}
\begingroup
  \ifGPblacktext
    \def\colorrgb#1{}%
    \def\colorgray#1{}%
  \else
    \ifGPcolor
      \def\colorrgb#1{\color[rgb]{#1}}%
      \def\colorgray#1{\color[gray]{#1}}%
      \expandafter\def\csname LTw\endcsname{\color{white}}%
      \expandafter\def\csname LTb\endcsname{\color{black}}%
      \expandafter\def\csname LTa\endcsname{\color{black}}%
      \expandafter\def\csname LT0\endcsname{\color[rgb]{1,0,0}}%
      \expandafter\def\csname LT1\endcsname{\color[rgb]{0,1,0}}%
      \expandafter\def\csname LT2\endcsname{\color[rgb]{0,0,1}}%
      \expandafter\def\csname LT3\endcsname{\color[rgb]{1,0,1}}%
      \expandafter\def\csname LT4\endcsname{\color[rgb]{0,1,1}}%
      \expandafter\def\csname LT5\endcsname{\color[rgb]{1,1,0}}%
      \expandafter\def\csname LT6\endcsname{\color[rgb]{0,0,0}}%
      \expandafter\def\csname LT7\endcsname{\color[rgb]{1,0.3,0}}%
      \expandafter\def\csname LT8\endcsname{\color[rgb]{0.5,0.5,0.5}}%
    \else
      \def\colorrgb#1{\color{black}}%
      \def\colorgray#1{\color[gray]{#1}}%
      \expandafter\def\csname LTw\endcsname{\color{white}}%
      \expandafter\def\csname LTb\endcsname{\color{black}}%
      \expandafter\def\csname LTa\endcsname{\color{black}}%
      \expandafter\def\csname LT0\endcsname{\color{black}}%
      \expandafter\def\csname LT1\endcsname{\color{black}}%
      \expandafter\def\csname LT2\endcsname{\color{black}}%
      \expandafter\def\csname LT3\endcsname{\color{black}}%
      \expandafter\def\csname LT4\endcsname{\color{black}}%
      \expandafter\def\csname LT5\endcsname{\color{black}}%
      \expandafter\def\csname LT6\endcsname{\color{black}}%
      \expandafter\def\csname LT7\endcsname{\color{black}}%
      \expandafter\def\csname LT8\endcsname{\color{black}}%
    \fi
  \fi
  \setlength{\unitlength}{0.0500bp}%
  \begin{picture}(4384.00,4533.33)%
    \gplgaddtomacro\gplbacktext{%
      \csname LTb\endcsname%
      \put(264,550){\makebox(0,0)[r]{\strut{} 0}}%
      \put(264,993){\makebox(0,0)[r]{\strut{} 1}}%
      \put(264,1436){\makebox(0,0)[r]{\strut{} 2}}%
      \put(264,1878){\makebox(0,0)[r]{\strut{} 3}}%
      \put(264,2321){\makebox(0,0)[r]{\strut{} 4}}%
      \put(264,2764){\makebox(0,0)[r]{\strut{} 5}}%
      \put(264,3207){\makebox(0,0)[r]{\strut{} 6}}%
      \put(264,3649){\makebox(0,0)[r]{\strut{} 7}}%
      \put(264,4092){\makebox(0,0)[r]{\strut{} 8}}%
      \put(396,330){\makebox(0,0){\strut{} 0}}%
      \put(1154,330){\makebox(0,0){\strut{} 3}}%
      \put(1911,330){\makebox(0,0){\strut{} 6}}%
      \put(2669,330){\makebox(0,0){\strut{} 9}}%
      \put(3426,330){\makebox(0,0){\strut{} 12}}%
      \put(4184,330){\makebox(0,0){\strut{} 15}}%
      \put(2290,110){\makebox(0,0){\strut{}Time t}}%
      \put(2290,4422){\makebox(0,0){\strut{}Statenorm evolution}}%
    }%
    \gplgaddtomacro\gplfronttext{%
      \csname LTb\endcsname%
      \put(3197,3919){\makebox(0,0)[r]{\strut{}$\|y_1(\cdot)\|_X$ (relaxed)}}%
      \csname LTb\endcsname%
      \put(3197,3699){\makebox(0,0)[r]{\strut{}$\|y_2(\cdot)\|_X$ (relaxed)}}%
      \csname LTb\endcsname%
      \put(3197,3479){\makebox(0,0)[r]{\strut{}$\|y_1(\cdot)\|_X$ (integer)}}%
      \csname LTb\endcsname%
      \put(3197,3259){\makebox(0,0)[r]{\strut{}$\|y_2(\cdot)\|_X$ (integer)}}%
    }%
    \gplbacktext{}
    \put(0,0){\includegraphics{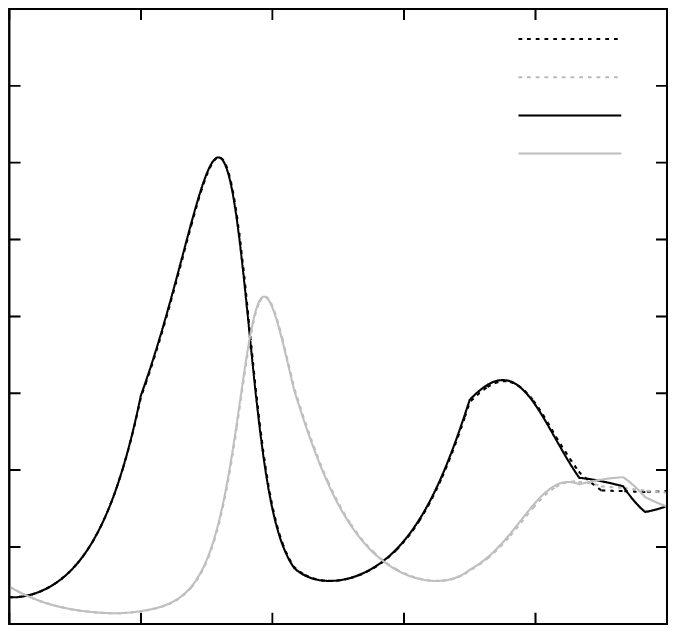}}%
    \gplfronttext{}
  \end{picture}%
\endgroup
}
\caption{Numerical results for Example~\ref{ex:lotkadiff2d}. The left figure shows the best found 
relaxed and integer control $\alpha^*(\cdot)$, $v^*(\cdot)$ and the right figure shows the corresponding 
evolutions of the populations $y_1(\cdot)$, $y_2(\cdot)$.}
\label{fig:lotkadiff2d} 
\end{figure}
\end{example}

\section{Conclusions and Open Problems}\label{sec:final}
We considered mixed-integer optimal control problems for abstract semilinear 
evolution equations and obtained conditions guaranteeing that the 
value function and the state of a relaxed optimal control problem can be 
approximated with arbitrary precision using a control that satisfies integer 
restrictions. In particular, our approach is constructive and gives rise to a 
numerical method for mixed-integer optimal control problems with certain partial
differential equations. Moreover, we showed how these conditions imply 
a-priori estimates on the quality of the solution when combinatorial 
constraints are enforced.  

We note that we did not discuss convergence of the constructed sequence of integer controls
approximating the optimal value of the relaxed problem.
It is not even clear in which topology such a convergence would be meaningful. Several
issues related to this questions, in particular in a PDE-context, is discussed in 
\cite{HanteLeugeringSeidman2009,HanteLeugeringSeidman2010} and \cite{FalkHanteDissertation2010}.

Compared to the previously available results on mixed-integer optimal control problems with
ordinary differential equations in \cite{SagerBockReinelt2009,SagerBockDiehl2011}, the setting 
treated in this paper involves a differential operator $A$, taken to be a generator 
of a strongly continuous semigroup. This requires careful regularity considerations.
When $A$ is a Laplace operator, we showed on a linear and a semilinear example
how such regularity assumptions can be met and provided numerical examples
demonstrating the practicability of the approach.

It is clear that the methodology considered in this paper generalizes to the case 
when the generator $A$ of a strongly continuous semigroup is replaced by a family
$\{A(t)\}_{t \in [0,t_f]}$ of unbounded linear operators generating an evolution
operator in the sense of \cite{Krein1971}. On the other hand it is not so clear how 
to extend the results in case of unbounded control action, for example, Neumann or 
Dirichlet boundary control for the heat equation. Recalling the density of solutions 
to \eqref{eq:di} in the set of solution to \eqref{eq:diconvex} 
which motivated our approach, we note that the case of unbounded control is not
covered by the available results on operator differential inclusions. While in 
principle semigroup techniques can deal with unbounded control operators, 
see for example the exposition in \cite[Chapter 3]{BensoussanDaPratoDelfourMitter1992}, 
this extension is non-trivial and requires additional work.

\section*{Acknowledgements}
{\footnotesize
Most of this research was carried out while both authors were member of the working group of Prof. H.G. Bock at the 
Interdisciplinary Center of Scientific Computing (IWR), University of Heidelberg. The financial support of the 
Mathematics Center Heidelberg (MATCH), of the Heidelberg Graduate School of Mathematical and Computational Methods 
for the Sciences (HGS MathComp), and of the EU project EMBOCON under grant FP7-ICT-2009-4 248940 is gratefully acknowledged. The
first author also acknowledges the support from Andreas Potschka with the software package MUSCOD-II.
}



\end{document}